\documentclass[11pt]{article}

\usepackage[margin=0.9in]{geometry}
\usepackage{setspace}
\usepackage{amsmath,amssymb,amsfonts,amsthm,mathtools,mathrsfs}
\usepackage{bm}
\usepackage{bbm}
\usepackage{booktabs}
\usepackage{graphicx}
\usepackage{float}
\usepackage{caption}
\usepackage{subcaption}
\usepackage{array}
\usepackage{enumitem}
\usepackage{natbib}
\usepackage{hyperref}
\usepackage{xcolor}
\usepackage{algorithm}
\usepackage{algpseudocode}
\usepackage{threeparttable}
\usepackage{longtable}
\usepackage{needspace}

\hypersetup{
  colorlinks=true,
  linkcolor=black,
  citecolor=black,
  urlcolor=black
}

\newtheorem{assumption}{Assumption}
\theoremstyle{definition}
\newtheorem{definition}{Definition}
\theoremstyle{plain}
\newtheorem{lemma}{Lemma}
\newtheorem{proposition}{Proposition}
\newtheorem{theorem}{Theorem}
\newtheorem{corollary}{Corollary}
\theoremstyle{remark}
\newtheorem{remark}{Remark}
\theoremstyle{plain}

\newcommand{\barF}{\overline F}
\newcommand{\calF}{\mathcal{F}}
\newcommand{\calQ}{\mathcal{Q}}

\newcommand{\calW}{\mathcal{W}}
\newcommand{\calH}{\mathcal{H}}

\newcommand{\dist}{\operatorname{dist}}
\newcommand{\CE}{\mathrm{CE}}
\newcommand{\PE}{\mathrm{PE}}
\newcolumntype{L}[1]{>{\raggedright\arraybackslash}p{#1}}

\begin{document}

\begingroup
\renewcommand{\thefootnote}{\fnsymbol{footnote}}

\begin{center}
{\fontsize{20.74}{22}\selectfont
Learning from Historical Transactions: Robust Supplier Pricing and Stocking with Sparse Data\par}
\vspace{1cm}

{Zhiqiang Chen\textsuperscript{a}}, Yi Cao\textsuperscript{a,*}, Caihua Chen\textsuperscript{a}\vspace{0.2cm}

{\fontsize{8}{10}\selectfont\itshape
\textsuperscript{a} School of Management and Engineering, Nanjing University, 210093, Nanjing, China
}\vspace{0.2cm}

\end{center}

\footnotetext[1]{Corresponding author
\par
\begingroup
\raggedright
\leftskip=1.8em
\noindent
\textit{Email address:}
\texttt{chenzq@smail.nju.edu.cn} (Zhiqiang Chen),
\texttt{yicao@smail.nju.edu.cn} (Yi Cao),
\texttt{chchen@nju.edu.cn} (Caihua Chen)
\par
\endgroup
}
\endgroup

\noindent
Upstream suppliers often set wholesale prices and stocking levels without direct access to detailed end-market demand information, which is instead held by downstream retailers. We study how a supplier can learn from a small number of historical transactions consisting of wholesale prices, retailer-chosen retail prices, and the resulting demand responses. We develop a robust, nonparametric framework that combines inverse optimization with distributionally robust optimization. Historical retailer decisions are used to restrict the set of demand distributions consistent with the observed transactions, after which the supplier optimizes against the least favorable demand response within this ambiguity set. We further accommodate imperfect retailer decisions through two error models (condition error and price error) and derive tractable worst-case response computations based on linear programs and one-dimensional price searches. Numerical experiments across six demand environments show that incorporating historical decision information can generate substantial profit improvements with only a few transactions and that these gains persist under moderate decision errors.


\section{Introduction}
\label{sec:intro}

Across many industries, upstream suppliers reach end customers through independent retailers, distributors, or service networks rather than through direct sales. Although such channel structures expand market reach, they also separate the supplier's pricing and stocking decisions from the detailed demand information generated at the point of sale. To set a wholesale price and stocking level, the supplier must nevertheless anticipate how both the retailer and the end market will respond to new contract terms. This task is particularly challenging because wholesale contracts are typically renegotiated only periodically, while retail prices may remain unchanged for extended periods owing to menu costs, internal approval procedures, and customer-facing considerations \citep{LevyEtAl1997,ZbarackiEtAl2004}. Consequently, even for an established product, the supplier may observe only a handful of distinct wholesale prices, downstream pricing responses, and corresponding sales outcomes. Such a sparse transaction history is generally insufficient to recover the end-market demand curve directly.

In many decentralized channels, much of the demand information unavailable to the supplier resides downstream. Retailers interact directly with consumers and accumulate granular knowledge about purchasing behavior, local demand conditions, and responses to promotions. This information asymmetry is particularly visible in grocery retail. A review by the UK Competition and Markets Authority documents that supermarkets use loyalty-program purchase histories to construct detailed shopper profiles, while consumer-goods suppliers obtain some of these insights only through retailer-controlled analytics services such as Tesco's Dunnhumby and Sainsbury's SupplyHub \citep{CMA2024LoyaltyPricing}. Geographic separation, organizational boundaries, and limited interfirm data sharing can therefore leave the supplier observing the wholesale terms of past transactions, the retail prices subsequently chosen, and the resulting sales outcomes, while the underlying demand information used by the retailer remains private. Prior contracting research shows that such private downstream demand information can materially affect contract design and channel performance \citep{CorbettZhouTang2004,KalkanciChenErhun2011,NasserTurcic2019}. Thus, even when the supplier cannot directly observe the retailer's demand information, the decisions made using that information may themselves provide an indirect and potentially valuable signal.

Existing approaches do not fully exploit this signal. Faced with limited visibility into end-market demand, a supplier may impose a parametric demand model, estimate demand nonparametrically, or protect its decisions against a set of demand curves consistent with the observed transaction outcomes \citep{BesbesZeevi2009,denBoer2015}. A parametric model can extrapolate from a small number of observations, but its recommendations may be sensitive to functional-form misspecification \citep{NambiarSimchiLeviWang2019}. Flexible nonparametric approaches avoid this restriction, but learning an unknown demand curve generally requires price experimentation and sufficiently rich variation in observed prices \citep{BesbesZeevi2009,CheungSimchiLeviWang2017}. Robust pricing based on partial quantile information offers a useful middle ground: it retains all demand curves consistent with the observed retail prices, purchase probabilities, and a mild shape restriction, and protects the decision against the least favorable curve in this set \citep{allouah2023optimal,GeHeWangWang2025}. Such methods can provide meaningful guarantees from only a few observations. However, they treat each historical retail price merely as a point at which demand was observed. In our setting, the retail price is itself an endogenous decision made by a better-informed retailer in response to a particular wholesale price. Ignoring that wholesale-price context therefore discards information about how the observed price--purchase pair was generated.

Motivated by this gap, we take a decision-informed view of historical transactions. A transaction records not only the market outcome associated with a retail price, but also a pricing decision made by a retailer with superior knowledge of end-market demand. Because this decision is made in response to the wholesale term offered by the supplier, the pair of wholesale and retail prices contains information about the retailer's underlying demand assessment that is not revealed by the sales outcome alone. Historical transactions can therefore be used not merely as observations of demand, but also as indirect evidence about which demand conditions could have rationalized the retailer's past behavior. This perspective raises two central questions for the supplier: how much additional demand information can be extracted from the retailer's historical pricing decisions, and how valuable is that information for subsequent wholesale pricing and stocking decisions, particularly when the retailer's decisions are not perfectly optimal? To address these questions, we develop a decision-informed (DI) approach that links what can be learned from historical transactions to the supplier's subsequent operational decisions.

Our contributions are as follows.

1. \textbf{Developing a robust, nonparametric framework via inverse optimization.} 
We develop a robust, nonparametric framework that integrates inverse optimization and distributionally robust optimization to infer demand information from historical transactions. By incorporating the decision context embedded in past retailer choices, our framework identifies demand curves that are consistent with both observed outcomes and historical decisions, leading to more reliable and less conservative supplier pricing and stocking decisions under demand uncertainty.

2. \textbf{Separating two sources of imperfect decisions.} Condition error (CE) models errors in the retailer’s optimality condition, allowing the observed price to violate the wholesale-price-implied optimality condition by a bounded amount. Price error (PE) models errors in the retailer’s decision, allowing the observed price to lie within a neighborhood of an unobserved exact optimum. We show how each error enlarges the set of plausible demand curves, when the historical decision remains informative, and how the two error measures are related. We also give explicit linear-program subproblems and a response-price search procedure for computing the resulting worst-case responses.

3. \textbf{Translating learned information into supplier actions.} For each candidate wholesale price, the information stage yields a lowest feasible demand across all admissible optimal retailer responses. This demand guarantee leads directly to robust stocking and contract choices. Numerical experiments across six demand environments demonstrate that DI improves on the benchmark method across the reported demand environments, achieving substantial profit improvements with only a few historical transactions and remaining effective under imperfect decision observations.

\subsection{Related Literature}
\label{subsec:literature}

Our study relates to three streams of literature: supply-chain information and transaction data, robust pricing with limited demand information, and distributionally robust optimization.

\vspace{-0.4cm}

\paragraph{Supply-chain information and transaction data.}

A large literature studies supply-chain contracting and pricing under asymmetric demand information, where downstream firms possess information that is unavailable to upstream suppliers \citep{CorbettZhouTang2004,AkanAtaLariviere2011,KalkanciChenErhun2011}. More recent work asks whether suppliers can learn from downstream actions in addition to observing realized demand. In newsvendor settings, historical order or reservation decisions can reveal information about the demand distribution and improve subsequent contract design \citep{ZhaoHaskellYu2024,YuDongSun2025,HuangYuZhao2026}. These studies demonstrate that transaction records contain information beyond realized demand because downstream decisions reflect private knowledge of market conditions.
Our setting differs in both the observed decision and the resulting information structure. Rather than observing newsvendor order or reservation quantities, the supplier observes retail prices chosen by an informed retailer in response to historical wholesale prices. These pricing decisions restrict the local shape of the end-market demand curve through the retailer's optimality condition. We use this information to refine the set of plausible demand distributions and support subsequent wholesale-pricing and stocking decisions.

\vspace{-0.4cm}

\paragraph{Robust pricing with partial demand information.}

Robust pricing has been studied when the decision maker has only partial information about demand, including support, moments, samples, or observed quantiles \citep{BergemannSchlag2011,AllouahBahamouBesbes2022,allouah2023optimal}. Closest to our demand model, \citet{GeHeWangWang2025} study robust pricing from a small number of price--purchase observations under $\alpha$-regularity and characterize worst-case demand without imposing a parametric distribution.
We build on this shape-constrained perspective but exploit an additional source of information. In conventional robust pricing, an observed price--demand pair identifies a point on the demand curve. In our setting, the retail price is itself an endogenous decision made by a better-informed retailer in response to a wholesale price. The associated wholesale term therefore provides information about which demand curves could have generated the observed pricing decision.

\vspace{-0.4cm}

\paragraph{Distributionally robust optimization.}

Distributionally robust optimization protects decisions against a set of plausible probability distributions. Classical ambiguity sets use support or moment information \citep{DelageYe2010}, while Wasserstein approaches construct neighborhoods around empirical distributions and provide tractable data-driven guarantees \citep{MohajerinEsfahaniKuhn2018,GaoKleywegt2023, chen2026joint}. Recent supply-chain studies further incorporate downstream optimality information into data-driven or Wasserstein ambiguity sets, showing that transaction decisions can refine distributional uncertainty when downstream firms possess superior demand information \citep{ZhaoHaskellYu2024,YuDongSun2025,HuangYuZhao2026}.
Our framework addresses a different data regime. Rather than constructing a neighborhood around an empirical distribution of demand realizations, we consider a small number of retail price--demand observations and impose a nonparametric $\alpha$-regularity restriction. Historical retailer optimality conditions then shrink this shape-constrained ambiguity set directly. This structure allows us to study how a small number of endogenous pricing decisions refine demand uncertainty and how that information translates into robust supplier pricing and stocking decisions.

\subsection{Organization}
\label{subsec:organization}

The remainder of this paper is organized as follows. Section~\ref{sec:model} introduces the market setting and shows how historical transactions restrict the set of demand distributions and retailer responses. Section~\ref{sec:noise} extends the framework to imperfect retailer decisions through CE and PE, develops the corresponding worst-case response computation, and studies when historical decision information becomes redundant. Section~\ref{sec:robust} translates the resulting demand guarantees into supplier pricing and stocking decisions. Section~\ref{sec:numerical} presents the numerical study, and Section~\ref{sec:conclusion} concludes.

\section{Model and Transaction-Informed Demand}
\label{sec:model}

The supplier typically sets the wholesale price and stocking level before observing the demand generated through the retailer. A higher wholesale price increases profit per unit sold, but may also induce the retailer to raise price and reduce demand. Understanding this tradeoff requires information about how the retailer’s pricing decisions respond to underlying demand conditions. In this section, we study how a short transaction history can reveal such information.

\subsection{Market Setting and Transaction Information}
\label{subsec:market-setting}

We consider a retailer that knows the distribution of consumer valuations and chooses a retail price to maximize its profit. Consumer values are characterized by a distribution $F$ supported on $\mathbb R_+$, with density $f$. At a retail price $s$, the purchase probability is $\bar F(s)=1-F(s)$. We normalize the market size to one, so demand at retail price $s$ is given by $\bar F(s)$. The retailer, who directly faces the end market, has access to the demand distribution $F$, whereas the supplier lacks such information. Given a wholesale price $w$, the retailer selects a retail price to maximize its profit:

\vspace{-0.4cm}

\begin{equation}
\mathcal S_F(w)=\arg\max_{s\ge w}(s-w)\barF(s).
\label{eq:retailer}
\end{equation}
The supplier observes $\mathcal H=\{(w_i,s_i,q_i):i=1,\ldots,T\}$, where $q_i=\bar F(s_i)$. After relabeling, we have $s_1<\cdots<s_T$, $1>q_1>\cdots>q_T>0$, and $s_i>w_i$. Each pair $(s_i,q_i)$ records the observed demand level at a given retail price, while the associated $w_i$ provides the wholesale-price context under which the retailer made that pricing decision. By convention, we set $\bar F(0)=1$, so that $(s_0,q_0)=(0,1)$.

Rather than imposing a specific functional form on demand, we maintain the following shape restriction.

\begin{definition}
\label{def:alpha-regular}
For $\alpha\in[0,1]$, define $\psi_\alpha(s\mid F)=(1-\alpha)s-\bar F(s)/f(s)$. Demand is \emph{$\alpha$-regular} if $\psi_\alpha(\cdot\mid F)$ is nondecreasing.
\end{definition}

This regularity condition unifies two widely studied classes of demand distributions. When $\alpha=0$, it reduces to Myerson's classical regularity condition, while $\alpha=1$ corresponds to the monotone hazard rate (MHR) condition \citep{myerson1981optimal,cole2014sample}. Both classes are widely used in data-driven and robust pricing models \citep{allouah2023optimal}. Moreover, many standard distribution families satisfy these shape restrictions. For example, exponential and uniform distributions are MHR, and broad classes of log-concave distributions also satisfy the MHR property \citep{bagnoli2005logconcave}. The supplier therefore uses this nonparametric shape restriction rather than specifying a demand family.

Define $\phi_F(s)=s-\bar F(s)/f(s)$. Since $\phi_F(s)=\psi_\alpha(s\mid F)+\alpha s$, $\phi_F$ is nondecreasing under $\alpha$-regularity. Therefore, for an exact historical decision at a price with $f(s_i)>0$, the retailer's optimal pricing condition can be written as
\begin{equation}
\phi_F(s_i)=w_i.
\label{eq:historical-foc}
\end{equation}
Under the maintained shape restriction, this condition identifies a global optimum of the retailer's pricing problem. The historical wholesale price therefore provides information beyond the observed demand point: it rules out demand curves that match the observed pair $(s_i,q_i)$ but cannot rationalize the retailer's pricing decision under $w_i$.

To isolate the value of the wholesale-price information, we compare two methods under the same shape restriction. Let $\mathcal F_\alpha^Q(\mathcal H)$ and $\mathcal F_\alpha^{DI}(\mathcal H)$ denote the ambiguity sets under the quantile-only and decision-informed methods, respectively. Quantile-only (Q) uses the observed price--demand pairs, whereas decision-informed (DI) additionally requires the demand distribution to rationalize the retailer's historical pricing decisions:

\vspace{-1cm}

\begin{align}
\calF_\alpha^Q(\calH)
&=
\left\{
F:
F\text{ is $\alpha$-regular},\
\bar F(s_i)=q_i,\ \forall i
\right\},
\label{eq:FQ}\\
\calF_\alpha^{DI}(\calH)
&=
\left\{
F\in\calF_\alpha^Q(\calH):
\phi_F(s_i)=w_i,\ \forall i
\right\}.
\label{eq:FDI}
\end{align}

\vspace{-0.6cm}

\begin{assumption}[Admissible demand class]
\label{assump:q-feasible}

Every distribution considered in the ambiguity sets satisfies the following conditions:

\noindent
1. \textbf{Absolute continuity.} The distribution is absolutely continuous with interval support.

\noindent
2. \textbf{Interior regularity.} The density is positive and locally absolutely continuous in the interior.

\noindent
3. \textbf{Tail condition.} $\lim_{s\to\infty}s\bar F(s)=0$.
\end{assumption}

For the exact decision-informed analysis in this section, we further assume that $\calF_\alpha^{DI}(\calH)\neq\emptyset$.

The supplier chooses a wholesale price $w$ from a finite set $\calW\subset\mathbb R_+$ and stock level $z\ge0$.
The supplier fulfills demand using regular stock $z$, supplemented by timely emergency sourcing when demand exceeds $z$. Regular and emergency fulfillment incur unit costs $c$ and $b$, respectively, with $0<c<b<w$.
Given induced demand $q$, its profit is

\vspace{-0.5cm}

\begin{equation}
\pi(w,z,q)
=
wq-cz-b(q-z)^+,
\label{eq:profit}
\end{equation}
where $(x)^+=\max\{x,0\}$. The three terms represent wholesale revenue, regular-stock cost, and emergency-sourcing cost, respectively. For $M\in\{Q,DI\}$, the supplier evaluates each candidate wholesale price against all demand distributions in $\mathcal F_\alpha^M(\mathcal H)$. For any admissible demand distribution $F$ and wholesale price $w$, let $r_F(w)\in\mathcal S_F(w)$ denote the response used in the supplier's robust evaluation, and write $q_F(w)=\bar F(r_F(w))$ for the induced demand. For a fixed $w\in\calW$, the resulting distributionally robust problem is
\begin{equation}
\Pi_\alpha^M(w)
=
\sup_{z\ge0}
\inf_{F\in\mathcal F_\alpha^M(\mathcal H)}
\pi\bigl(w,z,q_F(w)\bigr).
\label{eq:supplier-dro}
\end{equation}

The supplier then selects the wholesale price with the largest robust value:

\vspace{-0.3cm}

\begin{equation}
V_\alpha^M
=
\max_{w\in\calW}
\Pi_\alpha^M(w).
\label{eq:upstream-robust}
\end{equation}
Thus, the supplier first determines the robust stocking decision for each candidate wholesale price and then compares the resulting values across the finite menu.

\subsection{Demand Bounds from Historical Transactions}\label{subsec:demand-bounds}

A transaction history does not reveal demand at every price the retailer might choose next. This incomplete demand information creates uncertainty about the retailer's response to a new wholesale price. Following a distributionally robust approach, we therefore evaluate the supplier's decision against the worst-case demand distribution within the ambiguity sets constructed above. For $\alpha$-regular demand, existing studies show that demand bounds can be constructed from observed price--demand pairs using generalized-Pareto curves \citep{allouah2023optimal,GeHeWangWang2025}. We build on this characterization to derive the demand bounds implied by historical transactions under Q and DI.

The normalized generalized-Pareto family is

\vspace{-0.3cm}

\begin{equation}
\bar\Gamma_\alpha(t)=
\begin{cases}
[1+(1-\alpha)t]^{-1/(1-\alpha)},&0\le\alpha<1,\\
e^{-t},&\alpha=1,
\end{cases}
\label{eq:gamma-main}
\end{equation}
on its natural domain. For $0\le\alpha<1$, we extend $\bar\Gamma_\alpha(t)$ by setting $\bar\Gamma_\alpha(t)=+\infty$ for $t\le -1/(1-\alpha)$, with the convention $1/(+\infty)=0$. For observations $(a,q_a)$ and $(b,q_b)$ with $a<b$ and $q_a>q_b$, define

\vspace{-0.3cm}

\begin{equation}
\begin{gathered}
G_\alpha(v\mid a,q_a,b,q_b)=
\begin{cases}
1,&v<\ell_\alpha(a,q_a;b,q_b),\\
q_a\bar\Gamma_\alpha\!\left(
\bar\Gamma_\alpha^{-1}\!\left(\dfrac{q_b}{q_a}\right)\dfrac{v-a}{b-a}
\right),&v\ge\ell_\alpha(a,q_a;b,q_b).
\end{cases}
\end{gathered}
\label{eq:gpd-interpolation-main}
\end{equation}
where $\ell_\alpha(a,q_a;b,q_b)=a+(b-a)\frac{\bar\Gamma_\alpha^{-1}(1/q_a)}{\bar\Gamma_\alpha^{-1}(q_b/q_a)}$.
The curve in \eqref{eq:gpd-interpolation-main} interpolates the two observations. Let $G_{\alpha,i}$ denote the curve associated with $(s_i,q_i)$ and $(s_{i+1},q_{i+1})$. On $[s_i,s_{i+1}]$, this curve provides a lower bound on demand consistent with the two endpoints, while the neighboring observations constrain demand from above through the extensions of the adjacent generalized-Pareto curves. Accordingly, the Q demand bounds over the entire price range are
\begin{equation}
\underline G_\alpha^Q(v)=
\begin{cases}
q_i, & v=s_i,\quad i=0,\ldots,T,\\
G_{\alpha,i}(v), & s_i<v<s_{i+1},\quad i=0,\ldots,T-1,\\
0, & v>s_T,
\end{cases}
\label{eq:q-lower-envelope}
\end{equation}
and
\begin{equation}
\overline G_\alpha^Q(v)=
\begin{cases}
q_i, & v=s_i,\quad i=0,\ldots,T,\\
1, & 0<v<s_1,\quad T=1,\\
G_{\alpha,1}(v), & 0<v<s_1,\quad T\ge2,\\
\min\{G_{\alpha,i-1}(v),G_{\alpha,i+1}(v)\},
& s_i<v<s_{i+1},\quad i=1,\ldots,T-2,\\
G_{\alpha,T-2}(v),
& s_{T-1}<v<s_T,\quad T\ge2,\\
G_{\alpha,T-1}(v), & v>s_T.
\end{cases}
\label{eq:q-upper-envelope}
\end{equation}

DI further incorporates the optimality information contained in the retailer's historical pricing decisions. 
From \eqref{eq:historical-foc}, the observed decision $(w_i,s_i)$ implies $\bar F'(s_i)=-q_i/(s_i-w_i)$. Thus, in addition to fixing demand at the observed point $(s_i,q_i)$, the wholesale price $w_i$ also restricts the local shape of the demand curve around $s_i$. The corresponding decision-implied boundary can be expressed using the same generalized-Pareto family:

\vspace{-0.3cm}

\begin{equation}
H_\alpha(v\mid r,q,w)
=
\begin{cases}
1,
& v<\ell_\alpha^{H}(r,q,w),\\[1mm]
q\bar\Gamma_\alpha
\left(
\dfrac{v-r}{r-w}
\right),
& v\ge \ell_\alpha^{H}(r,q,w),
\end{cases}
\label{eq:historical-gp-boundary}
\end{equation}
where $\ell_\alpha^{H}(r,q,w) = r+(r-w)\bar\Gamma_\alpha^{-1}(1/q)$. Let $H_{\alpha,i}(v)=H_\alpha(v\mid s_i,q_i,w_i)$. Combining these decision-implied restrictions with the Q bound gives the DI upper boundary:
\begin{equation}
\overline G_\alpha^{DI}(v)=
\begin{cases}
q_i, & v=s_i,\quad i=0,\ldots,T,\\
\min\{\overline G_\alpha^Q(v),H_{\alpha,1}(v)\},
& 0<v<s_1,\\
\min\{\overline G_\alpha^Q(v),
H_{\alpha,i}(v),H_{\alpha,i+1}(v)\},
& s_i<v<s_{i+1},\quad i=1,\ldots,T-1,\\
\min\{\overline G_\alpha^Q(v),H_{\alpha,T}(v)\},
& v>s_T.
\end{cases}
\label{eq:di-upper-envelope}
\end{equation}
Hence, DI narrows the feasible demand set by excluding curves that match the observed price--demand pairs but are inconsistent with the retailer's historical pricing decisions under the corresponding wholesale prices.

The demand bounds above are meaningful only when they are compatible with the maintained $\alpha$-regularity restriction. For the exact Q/DI analysis in this section, the maintained condition $\calF_\alpha^{DI}(\calH)\neq\emptyset$ ensures that the historical transactions admit at least one admissible $\alpha$-regular demand distribution. The following proposition establishes that the constructed Q and DI bounds contain every admissible $\alpha$-regular demand distribution consistent with the corresponding historical information.

\begin{proposition}[Validity of historical demand bounds]
\label{prop:history-feasible}

Suppose Assumption~\ref{assump:q-feasible} holds. For every
$F\in\mathcal F_\alpha^Q(\mathcal H)$,

\vspace{-0.5cm}

\begin{equation}
\underline G_\alpha^Q(v)
\le
\bar F(v)
\le
\overline G_\alpha^Q(v).
\label{eq:q-bound-validity}
\end{equation}
Moreover, for every
$F\in\mathcal F_\alpha^{DI}(\mathcal H)$,

\vspace{-0.4cm}

\begin{equation}
\underline G_\alpha^Q(v)
\le
\bar F(v)
\le
\overline G_\alpha^{DI}(v)
\le
\overline G_\alpha^Q(v).
\label{eq:di-bound-validity}
\end{equation}

\end{proposition}

\vspace{-0.6cm}

\subsection{Worst-Case Retailer Responses}
\label{subsec:unified-response}

Consider a new wholesale price $w_0$ in \eqref{eq:supplier-dro}. For each $F\in\mathcal F_\alpha^M(\mathcal H)$, where $M\in\{Q,DI\}$, the retailer may choose a different optimal price, resulting in different demand levels. To illustrate how historical information restricts these responses, we first consider candidate retail prices between consecutive historical observations. For $i=1,\ldots,T-1$, define $\mathcal I_i(w_0)=[s_i,s_{i+1}]\cap(w_0,\infty)$.

Within each interval, the neighboring price--demand observations $(s_i,q_i)$ and $(s_{i+1},q_{i+1})$ provide local demand bounds. Under DI, the corresponding historical wholesale prices $w_i$ and $w_{i+1}$ impose additional restrictions through the retailer's optimality conditions. However, these historical restrictions alone do not determine whether a candidate price--demand pair $(r,q)$ can arise under the new wholesale price $w_0$.

For $(r,q)$ to be a feasible response, $r$ must also be optimal under $w_0$. Fix $r\in\mathcal I_i(w_0)$ and let $q$ denote the corresponding demand. Under $\alpha$-regularity, this current optimality condition implies the following candidate-response boundary:

\vspace{-0.3cm}

\[
H_\alpha(v\mid r,q,w_0)
= \min\left\{
  1,
q\,\bar\Gamma_\alpha
\left(
\frac{v-r}{r-w_0}
\right)\right\}.
\]
Any demand curve supporting $(r,q)$ must lie below this boundary while remaining consistent with the historical observations. In particular, consistency with the two neighboring observations requires

\vspace{-0.3cm}

\[
q_i\le H_\alpha(s_i\mid r,q,w_0),
\qquad
q_{i+1}\le H_\alpha(s_{i+1}\mid r,q,w_0).
\]
These inequalities imply the following lower bounds on the demand at $r$:

\vspace{-0.2cm}

\[
L_{\alpha,i}(r;w_0)
=
\frac{q_i}{
\bar\Gamma_\alpha\!\left(
\dfrac{s_i-r}{r-w_0}
\right)},
\qquad
R_{\alpha,i+1}(r;w_0)
=
\frac{q_{i+1}}{
\bar\Gamma_\alpha\!\left(
\dfrac{s_{i+1}-r}{r-w_0}
\right)}.
\]

The historical demand observations impose an additional lower restriction through the interpolation bound $G_{\alpha,i}(r)$. 
Combining the historical demand restriction with the two current-optimality restrictions gives the lower response bound of the feasible response region:

\vspace{-0.3cm}

\begin{equation}
\underline R_{\alpha,i}(r;w_0)
=
\max\left\{
G_{\alpha,i}(r),
L_{\alpha,i}(r;w_0),
R_{\alpha,i+1}(r;w_0)
\right\}.
\label{eq:response-lower-bound}
\end{equation}

The three terms in \eqref{eq:response-lower-bound} have distinct roles. 
The first is determined by the historical price--demand observations. 
The latter two ensure that the demand curve can support $r$ as an optimal response to $w_0$ while remaining consistent with the neighboring historical observations. 
Together, they impose a necessary lower bound on demand at a candidate optimal price. A feasible response must also admit a common demand curve satisfying the complete history.

The upper boundary is inherited from the historical demand bounds in Section~\ref{subsec:demand-bounds}. 
Under Q, the candidate demand is bounded above by $\overline G_{\alpha,i}^{Q}(r)$. 
DI further tightens this bound by incorporating the optimality information contained in the historical wholesale prices. 
Thus, the historical information determines the set of admissible demand curves, while the current optimality condition determines which points on those curves can arise as retailer responses to $w_0$.

\begin{figure}[htbp]
\centering
\includegraphics[width=0.98\textwidth]{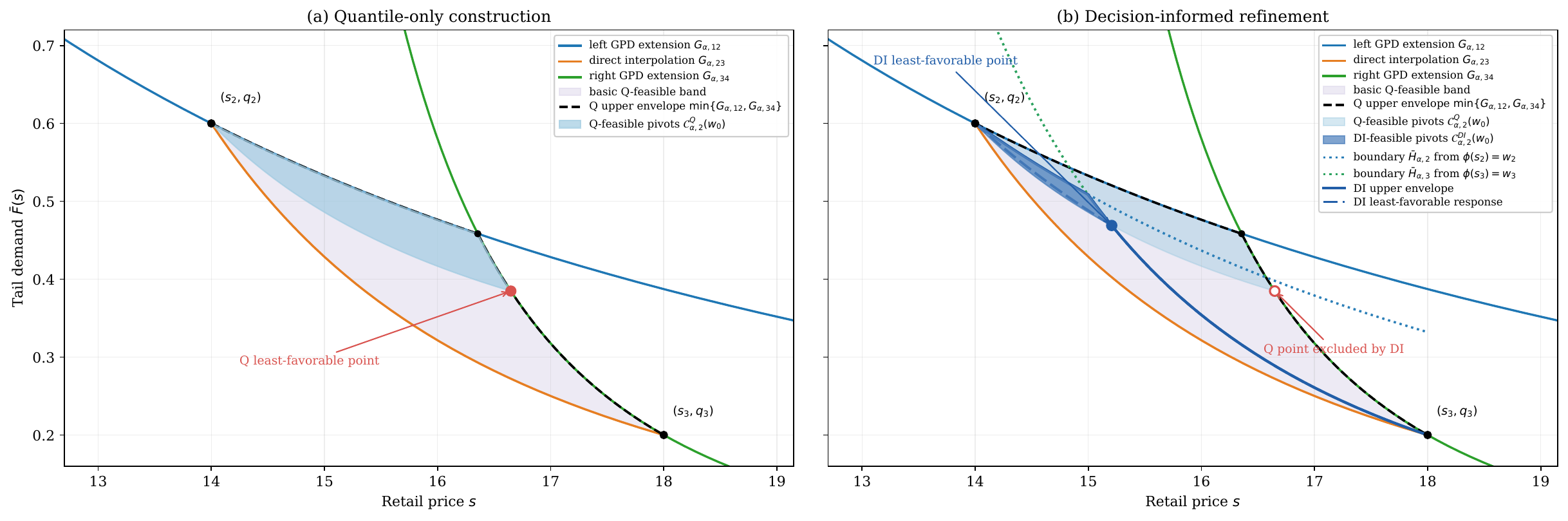}
\caption{Illustration of feasible retailer responses and worst-case demand under Q and DI.}
\label{fig:alpha-positive-response}
\end{figure}

Figure~\ref{fig:alpha-positive-response} illustrates this construction. 
The lower boundary combines the historical interpolation bound with the current-optimality restrictions in \eqref{eq:response-lower-bound}. 
Q retains the larger feasible response region based on the observed price--demand pairs, whereas DI further removes responses that are inconsistent with the retailer's historical pricing decisions. 

The preceding construction illustrates the demand restrictions within a historical price interval. For the supplier's robust evaluation, however, all optimal retailer responses must be considered. We therefore define the global feasible response set as

\vspace{-0.3cm}

\begin{equation}
\mathcal C_\alpha^M(w_0)
=
\left\{
(r,\bar F(r)):
F\in\mathcal F_\alpha^M(\mathcal H),\
r\in\mathcal S_F(w_0)
\right\}.
\label{eq:response-set}
\end{equation}
Hence, $\mathcal C_\alpha^M(w_0)$ collects all price--demand pairs that can be generated by an optimal retailer response under some demand distribution retained by information model $M$.

The response set above characterizes the outcomes relevant to the supplier's problem in \eqref{eq:supplier-dro}. The following theorem links this set to the supplier's worst-case evaluation.

\begin{theorem}[Worst-case retailer response]
\label{thm:insertion}

Fix $\alpha\in[0,1]$, $M\in\{Q,DI\}$, and $w_0\in\mathcal W$.
For every $z\ge0$,
\begin{equation}
\inf_{F\in\mathcal F_\alpha^M(\mathcal H)}
\pi\bigl(w_0,z,q_F(w_0)\bigr)
=
w_0q_{\min,\alpha}^M(w_0)-cz-b\bigl(q_{\min,\alpha}^M(w_0)-z\bigr)^+,
\label{eq:positive-worst-profit}
\end{equation}
where
\begin{equation}
q_{\min,\alpha}^M(w_0)
:=
\inf_{F\in\mathcal F_\alpha^M(\mathcal H)}
\bar F(r_F(w_0))
=
\inf_{(r,q)\in\mathcal C_\alpha^M(w_0)}q.
\label{eq:worst-demand}
\end{equation}

\end{theorem}

Theorem~\ref{thm:insertion} reduces the supplier's distributionally robust evaluation of a wholesale price to the single demand guarantee $q_{\min,\alpha}^M(w_0)$. DI can improve this guarantee by excluding demand distributions that remain feasible under Q but are inconsistent with the historical pricing decisions. A unified computational procedure for evaluating this demand guarantee is developed in Section~\ref{subsec:error-computation}.

\section{Imperfect Retailer Decisions}
\label{sec:noise}

Exact DI assumes that historical retail prices are optimal responses to their corresponding wholesale prices, which may be restrictive when retailer pricing is imperfect. We therefore extend the framework through two error models, CE and PE. Historical decision errors enlarge the ambiguity set, whereas errors in the retailer's response to a new contract affect the induced demand. We analyze both effects and develop a unified method for computing worst-case responses.

\begin{assumption}[Interior pricing regime]
\label{assump:interior-regime}
For each admissible demand distribution, every exact optimal retail price considered in the analysis lies in the interior of its support, where the density is strictly positive.
\end{assumption}

This setting reflects the usual pricing situation in which the retailer chooses its price by balancing a higher per-unit margin against the associated loss in demand. In contrast, a price located exactly at a support boundary represents a special case in which the retailer's decision is effectively pinned by the lowest or highest possible consumer valuation rather than by this economic tradeoff. Such support-driven pricing is not the focus of our analysis.

\subsection{Condition Error (CE)}
\label{subsec:foc-error}

CE models a bounded discrepancy between the retailer's pricing condition $\phi_F(s_i)$ and the historical wholesale price $w_i$, with tolerance $\delta_i$. Such discrepancies may arise from approximate optimization or imperfect assessment of wholesale-price incentives. 

Let $\boldsymbol{\delta}=(\delta_1,\ldots,\delta_T)$. The resulting ambiguity set is

\vspace{-0.3cm}

\begin{equation}
\calF_\alpha^{\CE,\boldsymbol\delta}(\calH)
=
\left\{
F\in\calF_\alpha^Q(\calH):
|\phi_F(s_i)-w_i|\le\delta_i,\ \forall i
\right\}.
\label{eq:foc-error-set}
\end{equation}
For the CE analysis, we restrict attention to tolerance vectors $\boldsymbol\delta$ for which $\calF_\alpha^{\CE,\boldsymbol\delta}(\calH)\neq\emptyset$.

CE therefore interpolates between exact DI and Q. If $\boldsymbol\delta\le\boldsymbol\delta'$ componentwise, then
$
\calF_\alpha^{DI}(\calH)
\subseteq
\calF_\alpha^{\CE,\boldsymbol\delta}(\calH)
\subseteq
\calF_\alpha^{\CE,\boldsymbol\delta'}(\calH)
\subseteq
\calF_\alpha^Q(\calH)
$.
A larger historical tolerance admits more demand distributions and therefore weakens the information extracted from the retailer's past pricing decisions.

The effect of CE on the historical demand bounds can be characterized using the boundary function in \eqref{eq:historical-gp-boundary}. At transaction $i$, the pricing condition may correspond to any effective wholesale-price level
$\vartheta\in[w_i-\delta_i,w_i+\delta_i]$. Hence the largest demand boundary compatible with the CE tolerance is

\vspace{-0.4cm}

\begin{equation}
H_{\alpha,i}^{\CE}(v;\delta_i)
=
\sup_{\vartheta\in[w_i-\delta_i,w_i+\delta_i]\cap(-\infty,s_i)}
H_\alpha(v\mid s_i,q_i,\vartheta).
\label{eq:ce-boundary-main}
\end{equation}
The restriction $\vartheta < s_i$ follows from $\phi_F(s_i) = s_i - q_i/f(s_i) < s_i$. Combining the Q upper bound with the neighboring CE boundaries gives the corresponding CE demand band. As $\delta_i$ increases, these decision-based restrictions become weaker, and CE reduces to Q once all historical decision restrictions become redundant.

CE can also affect the retailer's response to a new wholesale price. Given a contract $w_0$, suppose that the retailer's pricing condition is evaluated at an effective target
$\theta_0\in[w_0-\epsilon_0,w_0+\epsilon_0]$. Define

\vspace{-0.5cm}

\[
\mathcal S_F(\theta_0;w_0)
=
\arg\max_{s\ge w_0}(s-\theta_0)\bar F(s).
\]
Let $r_F(\theta_0;w_0)\in\mathcal S_F(\theta_0;w_0)$ denote the response used in the supplier's robust evaluation. The supplier continues to receive the contractual wholesale price $w_0$; the error affects only the retailer's pricing response. The following lemma identifies the worst realization of the current CE.

\begin{lemma}[Worst current CE realization]
\label{lem:current-target-monotonicity}

For every $\alpha\in[0,1]$, $\bar F(r_F(\theta_0;w_0))$ is nonincreasing in $\theta_0$ and the least favorable realization $\theta_0^{\mathrm{wc}} = w_0+\epsilon_0$.
\end{lemma}

The historical and current CE components can therefore be summarized by a single worst-case demand guarantee:

\vspace{-0.3cm}

\begin{equation}
q_{\min,\alpha}^{\CE,\boldsymbol\delta,\epsilon_0}(w_0)
=
\inf_{F\in\calF_\alpha^{\CE,\boldsymbol\delta}(\calH)}
\bar F\!\left(r_F(w_0+\epsilon_0;w_0)\right),
\qquad 0\le\alpha\le1.
\label{eq:ce-qmin-unified}
\end{equation}
The two error parameters play different roles: $\boldsymbol\delta$ determines how much information remains in the historical decisions, whereas $\epsilon_0$ determines how adverse the retailer's response to the new contract may be.

\vspace{-0.2cm}

\subsection{Price Error (PE)}
\label{subsec:decision-space-error}

PE allows the observed retail price to deviate from an exact optimum, potentially due to price-adjustment frictions, rounding, or imperfect implementation.

Let $\dist(s_i,\mathcal S_F(w_i))$ denote the distance between the observed retail price and the nearest exact optimal price under wholesale price $w_i$. With transaction-specific tolerances $\boldsymbol\eta=(\eta_1,\ldots,\eta_T)$, the resulting ambiguity set is

\vspace{-0.4cm}

\begin{equation}
\calF_\alpha^{\PE,\boldsymbol\eta}(\calH)
=
\left\{
F\in\calF_\alpha^Q(\calH):
\dist(s_i,\mathcal S_F(w_i))\le\eta_i,\ \forall i
\right\}.
\label{eq:decision-space-set}
\end{equation}
For the PE analysis, we restrict attention to tolerance vectors $\boldsymbol\eta$ for which $\calF_\alpha^{\PE,\boldsymbol\eta}(\calH)\neq\emptyset$.

Thus, a demand distribution remains feasible under PE as long as it matches the observed price--demand pairs and admits an exact optimal price within $\eta_i$ of each observed retail price. Importantly, the observed demand $q_i=\bar F(s_i)$ remains associated with $s_i$ and is not reassigned to the unobserved optimum.

PE can also affect the retailer's response to a new wholesale price. For a candidate contract $w_0$, let
$
\mathcal A_F(w_0;\eta_0)
=
\left\{
y\ge w_0:
\dist(y,\mathcal S_F(w_0))\le\eta_0
\right\}
$
denote the set of permitted implemented prices. Since demand is nonincreasing in the retail price, the least favorable implemented price under the current PE radius is $y_F(w_0;\eta_0)=r_F(w_0)+\eta_0$. The corresponding worst-case demand is therefore

\vspace{-0.3cm}

\begin{equation}
q_{\min,\alpha}^{\PE,\boldsymbol\eta,\eta_0}(w_0)
=
\inf_{F\in\calF_\alpha^{\PE,\boldsymbol\eta}(\calH)}
\bar F\!\left(y_F(w_0;\eta_0)\right).
\label{eq:decision-space-qmin}
\end{equation}
Historical PE therefore determines which demand distributions remain plausible, whereas the current radius $\eta_0$ determines how far the implemented retail price may deviate from an exact optimal response.

Under PE, the permitted price interval must contain at least one exact optimal retail price. Let $\underline{s}_i=s_i-\eta_i$ and $\overline{s}_i=s_i+\eta_i$. Since $\phi_F$ is nondecreasing, this requirement can be characterized by the pricing conditions at the interval endpoints. When the relevant endpoints lie in the positive-density support interior, we have

\vspace{-0.3cm}

\begin{equation}
\mathcal S_F(w_i)\cap[\underline{s}_i,\overline{s}_i]\neq\varnothing
\quad\Longleftrightarrow\quad
\begin{cases}
\phi_F(\underline{s}_i)\le w_i\le\phi_F(\overline{s}_i),
& \underline{s}_i>w_i,\\[1mm]
w_i\le\phi_F(\overline{s}_i),
& \underline{s}_i\le w_i.
\end{cases}
\label{eq:pe-endpoint-crossing-main}
\end{equation}
When $\underline{s}_i\le w_i$, the left-endpoint condition is unnecessary because the retailer only considers prices above the wholesale price. For distributions with bounded support, the permitted interval may extend beyond the support, where $\phi_F$ is undefined. In such cases, the corresponding endpoint condition is automatically satisfied and need not be imposed.

Unlike CE, PE leaves the pricing target unchanged at $w_0$ but allows the implemented retail price to deviate from an exact optimum. The supplier therefore evaluates demand at $y_F(w_0;\eta_0)$. Increasing the historical PE tolerances enlarges the ambiguity set, while increasing the current tolerance allows less favorable demand responses. Consequently, if $\boldsymbol\eta\le\boldsymbol\eta'$ and $\eta_0\le\eta_0'$, then
\begin{equation}
\calF_\alpha^{\PE,\boldsymbol\eta}(\calH)
\subseteq
\calF_\alpha^{\PE,\boldsymbol\eta'}(\calH),
\qquad
q_{\min,\alpha}^{\PE,\boldsymbol\eta,\eta_0}(w)
\ge
q_{\min,\alpha}^{\PE,\boldsymbol\eta',\eta_0'}(w).
\label{eq:decision-space-nesting}
\end{equation}

When exact DI is feasible, both CE and PE reduce to exact DI when the historical and current error tolerances are zero. Thus, for every $\alpha\in[0,1]$,
$
q_{\min,\alpha}^{\PE,\bm0,0}(w_0)
=
q_{\min,\alpha}^{\CE,\bm0,0}(w_0)
=
q_{\min,\alpha}^{DI}(w_0)
$.

\subsection{Computing the Worst-Case Response}
\label{subsec:error-computation}

The CE and PE models enlarge the set of demand distributions and retailer responses in different ways, but the resulting worst-case demand can be computed within a common framework. For a fixed candidate exact optimal price $r$, we compute the lowest demand that can be supported by an admissible demand distribution under the corresponding error model. We then use a one-dimensional grid search over $r$ to obtain the worst-case demand within a prescribed error tolerance.

To formulate the fixed-$r$ problem, let $\rho=1-\alpha$ and $u_i=\bar\Gamma_\alpha^{-1}(q_i)$. The two error models differ only in how the candidate exact optimal price $r$ is interpreted. Under CE, the worst current realization shifts the retailer's effective pricing target to $w_0+\epsilon_0$, while the implemented price remains $r$. Under PE, the pricing target remains $w_0$, but the implemented price can increase to $r+\eta_0$.

To write the two models jointly, define the effective pricing target $\theta_M$ and the implementation adjustment $d_M$ by
\begin{equation}
(\theta_M,d_M)
=
\begin{cases}
(w_0+\epsilon_0,0), & M=\CE,\\
(w_0,\eta_0), & M=\PE.
\end{cases}
\label{eq:error-current-parameters}
\end{equation}
For each candidate optimal price $r>\theta_M$, let $y=r+d_M$ be the implemented price. The fixed nodes consist of the historical prices $s_i$, together with $\overline{s}_i$ and those $\underline{s}_i>w_i$ under PE. Adding $r$ and $y$ and merging duplicate nodes gives the ordered grid $t_1<\cdots<t_m$, where $\Delta_j=t_{j+1}-t_j$.

We use the transformation

\vspace{-0.4cm}

\[
U(t)=\bar\Gamma_\alpha^{-1}(\bar F(t)),
\qquad
\bar F(t)=\bar\Gamma_\alpha(U(t)),
\]
and let $g(t)$ denote the derivative of $U$ at $t$. Under this transformation, $\alpha$-regularity is equivalent to convexity of $U$, while the retailer's optimality conditions become linear in $U$ and $g$ at fixed price locations. Write $U_j=U(t_j)$ and $g_j=g(t_j)$. For every pair of adjacent grid points, convexity requires

\vspace{-0.4cm}

\[
\Delta_jg_j
\le
U_{j+1}-U_j
\le
\Delta_jg_{j+1}.
\]

When either inequality is forced to equality, interpolation consistency requires the entire interval to be affine. In the original demand space, this means that if the candidate response attains a generalized-Pareto boundary at an additional interior point, the demand curve must coincide with that boundary between the corresponding contact points. We therefore let $\mathcal J_M(r)$ denote the set of grid intervals forced to be affine by the historical observations, the error restrictions, and the candidate response $r$. For these intervals, we impose

\vspace{-0.3cm}

\[
\Delta_jg_j
=
U_{j+1}-U_j
=
\Delta_jg_{j+1},
\qquad j\in\mathcal J_M(r).
\]

\begin{theorem}\label{thm:decision-space-reduction}

Under Assumption~\ref{assump:interior-regime}, fix $\alpha\in[0,1]$ and $M\in\{\CE,\PE\}$. For a candidate $r>\theta_M$, let $y=r+d_M$, and let $t_1<\cdots<t_m$ denote the complete ordered grid defined above. For each candidate $r$ that passes the interpolation-consistency test, define

\vspace{-0.9cm}

\begin{subequations}
\label{eq:error-common-lp}
\begin{align}
V_M(r)
=\sup_{U,g\ge0}\quad & U(y)
\label{eq:error-lp-objective}\\
\text{s.t.}\quad
&U(s_i)=u_i,
&& i=1,\ldots,T,
\label{eq:error-lp-data}\\
&\Delta_jg_j
\le U_{j+1}-U_j
\le\Delta_jg_{j+1},
&&j=1,\ldots,m-1,
\label{eq:error-lp-shape}\\
&U_1\le t_1g_1,
\label{eq:error-lp-support}\\
&(r-\theta_M)g(r)-\rho U(r)=1,
\label{eq:error-lp-current}\\
&(s_i-w_i-\delta_i)g(s_i)
\le1+\rho u_i
\le(s_i-w_i+\delta_i)g(s_i),
&&M=\CE,\ \forall i,
\label{eq:error-lp-ce}\\
&(\underline{s}_i-w_i)g(\underline{s}_i)
\le1+\rho U(\underline{s}_i),
&&M=\PE,\ \underline{s}_i>w_i,
\label{eq:error-lp-pe-left}\\
&(\overline{s}_i-w_i)g(\overline{s}_i)
\ge1+\rho U(\overline{s}_i),
&&M=\PE,\ \forall i,
\label{eq:error-lp-pe-right}\\
&\Delta_jg_j
=
U_{j+1}-U_j
=
\Delta_jg_{j+1},
&&j\in\mathcal J_M(r).
\label{eq:error-lp-flat}
\end{align}
\end{subequations}

\vspace{-0.3cm}

Let $\mathcal R_M$ denote the set of candidates $r>\theta_M$ that pass the interpolation-consistency test. Then the worst-case demand is

\vspace{-0.3cm}

\begin{equation}
q_{\min,\alpha}^M(w_0)
=
\bar\Gamma_\alpha
\left(
\sup_{r\in\mathcal R_M}V_M(r)
\right),
\qquad
M\in\{\CE,\PE\}.
\label{eq:decision-space-response-set}
\end{equation}

\end{theorem}

The formulation has a direct interpretation. Because $\bar\Gamma_\alpha(\cdot)$ is decreasing, maximizing $U(y)$ is equivalent to minimizing demand at the implemented retail price $y$. Constraint~\eqref{eq:error-lp-data} preserves the historical demand observations, while \eqref{eq:error-lp-shape} imposes the convexity associated with $\alpha$-regularity. Constraint~\eqref{eq:error-lp-support} ensures that the transformed demand curve is consistent with the nonnegative support of consumer values and $\bar F(0)=1$. Constraint~\eqref{eq:error-lp-current} makes $r$ optimal under the current effective pricing target. The CE constraints~\eqref{eq:error-lp-ce} allow bounded deviations in the historical optimality conditions, whereas the PE constraints~\eqref{eq:error-lp-pe-left}--\eqref{eq:error-lp-pe-right} require each permitted price neighborhood to contain an exact optimum. Finally, the interpolation-consistency test identifies intervals that must be affine, and constraint~\eqref{eq:error-lp-flat} enforces the corresponding equalities. The same LP framework also covers Q by omitting the historical decision constraints, and exact DI by setting all error tolerances to zero.

\paragraph{Searching over the exact optimal price.}

Theorem~\ref{thm:decision-space-reduction} provides an exact characterization of the worst-case demand through a fixed-$r$ LP and an outer optimization over the retailer's optimal price $r$. The remaining challenge is to carry out this continuous price search using finitely many LP evaluations while ensuring a prescribed level of accuracy.

We first establish a finite search range. Appendix~\ref{app:e4-range} derives an upper bound on $r$ when $\alpha>0$, while for $\alpha=0$, an auxiliary feasibility test either identifies zero worst-case demand or provides a finite bound. Within this range, Appendix~\ref{app:e4-partition} shows that feasible response prices form an interval, possibly a single point, and constructs a finite grid with a prescribed demand accuracy.

\begin{algorithm}[H]
\caption{Grid-based worst-case demand computation}
\label{alg:error-global-search}
\begin{algorithmic}[1]
\Require History $\calH$, $\alpha$, $w_0$, model $M$,
error tolerances, and accuracy $e>0$, with
$\mathcal R_M\neq\varnothing$.
\State Apply Algorithm~\ref{alg:e4-range-test} to obtain a finite range or a zero-demand certificate.
\If{the zero-demand test succeeds}
    \State \Return $\widehat q=0$.
\EndIf
\State Construct $\mathcal G_M(e)$ on $(\theta_M,\overline r_M]$ using Algorithm~\ref{alg:e4-grid-build}.
\State Initialize $\widehat q=1$.
\For{each $r\in\mathcal G_M(e)$}
    \State Solve the fixed-$r$ LP
    \eqref{eq:error-common-lp} for $V_M(r)$.
    \State Update
    $\widehat q=\min\{\widehat q,\bar\Gamma_\alpha(V_M(r))\}$.
\EndFor
\State \Return $\widehat q$.
\end{algorithmic}
\end{algorithm}

We summarize the grid-search procedure in Algorithm~\ref{alg:error-global-search} and establish its approximation error bound in Lemma~\ref{lem:e4-grid-error}.

\begin{lemma}[Finite-accuracy demand computation]
\label{lem:e4-grid-error}
Under the conditions of Theorem~\ref{thm:decision-space-reduction}, suppose $\mathcal R_M\neq\varnothing$. For any prescribed accuracy $e>0$, Algorithm~\ref{alg:error-global-search} evaluates finitely many candidate retail prices and returns an estimate $\widehat q_{\min,\alpha}^M(w_0)$ satisfying
\begin{equation}
0\le
\widehat q_{\min,\alpha}^M(w_0)
-q_{\min,\alpha}^M(w_0)
\le e.
\label{eq:e4-demand-error}
\end{equation}
\end{lemma}

\subsection{Informativeness of Historical Decisions}
\label{subsec:error-informativeness}

We now examine when a historical wholesale price ceases to provide information beyond the observed price--demand pairs. Specifically, we identify the critical CE and PE tolerances at which the corresponding historical decision restriction becomes redundant relative to the Q ambiguity set.

\begin{proposition}[CE and PE informativeness thresholds]
\label{prop:error-informativeness}

Under Assumption~\ref{assump:interior-regime},
for each historical observation $i$, the critical CE and PE tolerances are

\vspace{-0.8cm}

\begin{subequations}
\label{eq:critical-tolerances}
\begin{align}
\delta_i^{\mathrm{crit}}
&=
\sup_{F\in\calF_\alpha^Q(\calH)}
\bigl|\phi_F(s_i)-w_i\bigr|,
\label{eq:critical-error-main}\\
\eta_i^{\mathrm{crit}}
&=
\sup_{F\in\calF_\alpha^Q(\calH)}
\dist\!\left(s_i,\mathcal S_F(w_i)\right).
\label{eq:decision-distance-threshold}
\end{align}
\end{subequations}

\end{proposition}

The CE threshold is finite for every $\alpha\in[0,1]$. The PE threshold is also finite when $\alpha>0$. When $\alpha=0$, however, the PE threshold may be infinite, in which case the historical decision restriction remains informative for every finite tolerance. For any finite critical value, the corresponding restriction becomes redundant once the tolerance reaches that value.

The two thresholds capture different aspects of decision informativeness. The CE threshold measures the largest discrepancy in the retailer's pricing condition compatible with the observed price--demand pairs, whereas the PE threshold measures the largest distance from the observed retail price to its nearest exact optimum. Within each model, a larger threshold indicates that the historical wholesale price remains informative under greater decision errors. Both thresholds can be computed from the optimization problems in Proposition~\ref{prop:error-informativeness}. Appendix~\ref{app:threshold-computation} provides the corresponding formulations.

\subsection{Comparison of CE and PE}
\label{subsec:error-comparison}

CE and PE relax exact retailer optimality in different ways. CE bounds the discrepancy in the pricing condition at the observed price, whereas PE bounds the distance between the observed price and an exact optimum. The relationship between these two error measures depends on how rapidly $\phi_F$ changes with price. When $\alpha>0$, $\alpha$-regularity provides a lower bound on this rate of change, while the reverse conversion requires an upper bound. The following proposition establishes these relationships.

\begin{proposition}[CE--PE comparison]
\label{prop:error-model-comparison}

For $\alpha>0$,

\vspace{-0.3cm}

\begin{equation}
\mathcal F_\alpha^{\CE,\delta}(\mathcal H)
\subseteq
\mathcal F_\alpha^{\PE,\delta/\alpha}(\mathcal H),
\label{eq:foc-to-ds-set}
\end{equation}
where the division is componentwise.

For any $\alpha\in[0,1]$, suppose that the relevant exact optimal prices are interior and that $\phi_F$ is uniformly $L$-Lipschitz on the relevant price range for every distribution under consideration. Then

\vspace{-0.3cm}

\begin{equation}
\mathcal F_\alpha^{\PE,\eta}(\mathcal H)
\subseteq
\mathcal F_\alpha^{\CE,L\eta}(\mathcal H).
\label{eq:ds-to-foc-set}
\end{equation}

\end{proposition}

The first inclusion shows that a CE tolerance $\delta_i$ implies a price deviation of at most $\delta_i/\alpha$ from an exact optimum. The condition $\alpha>0$ is essential because regularity alone does not guarantee a positive lower bound on the rate of change of $\phi_F$, so a finite CE-to-PE conversion need not exist when $\alpha=0$. Conversely, under the Lipschitz condition, a PE radius $\eta_i$ implies a pricing-condition discrepancy of at most $L\eta_i$. These relationships provide sufficient conversions between the two error measures.

\section{Supplier Pricing and Stocking}
\label{sec:robust}

The preceding analysis characterizes the worst-case demand $q_{\min,\alpha}^M(w)$ for each wholesale price $w$ under $M\in\{Q,DI,\CE,\PE\}$. Theorem~\ref{thm:insertion} establishes this demand guarantee for Q and DI, while Theorem~\ref{thm:decision-space-reduction} extends the computation to CE and PE. We now use these results to determine the supplier's robust pricing and stocking decisions.

Since $w>b$, the supplier's profit is increasing in demand. Thus, for each $w\in\calW$, the worst-case profit is determined by $q_{\min,\alpha}^M(w)$, and the supplier's problem becomes
\[
\Pi_\alpha^M(w)
=
\sup_{z\ge0}
\left\{
wq_{\min,\alpha}^M(w)-cz
-b\bigl(q_{\min,\alpha}^M(w)-z\bigr)^+
\right\}.
\]

\begin{theorem}[Robust supplier policy]
\label{thm:capacity}

For every $\alpha\in[0,1]$ and model $M$, the optimal stock level for each $w\in\calW$ is
\begin{equation}
z_\alpha^{M,*}(w)=q_{\min,\alpha}^M(w),
\qquad
\Pi_\alpha^M(w)=(w-c)q_{\min,\alpha}^M(w).
\label{eq:value-fixed-w}
\end{equation}
Consequently,
\begin{equation}
V_\alpha^M
=
\max_{w\in\calW}
(w-c)q_{\min,\alpha}^M(w).
\label{eq:upstream-policy}
\end{equation}

\end{theorem}

Theorem~\ref{thm:capacity} shows that the optimal stock level equals the guaranteed demand under each wholesale price. The supplier then selects the contract that maximizes robust profit over the finite menu $\calW$. Historical decision information can improve these decisions by reducing demand uncertainty. Since $\calF_\alpha^{DI}(\calH)\subseteq\calF_\alpha^Q(\calH)$, the resulting demand and profit guarantees satisfy the following relationships.

\begin{corollary}[Value of historical decisions]
\label{cor:info-value}

For every $\alpha\in[0,1]$ and $w\in\calW$,
\begin{equation}
q_{\min,\alpha}^{DI}(w)
\ge
q_{\min,\alpha}^{Q}(w),
\qquad
\Pi_\alpha^{DI}(w)
\ge
\Pi_\alpha^Q(w),
\qquad
V_\alpha^{DI}
\ge
V_\alpha^Q.
\label{eq:qmin-info}
\end{equation}

\end{corollary}

In contrast, increasing decision-error tolerances weakens the worst-case demand guarantee. For CE, if
$\boldsymbol\delta\le\boldsymbol\delta'$ and $\epsilon_0\le\epsilon_0'$, then
\[
q_{\min,\alpha}^{\CE,\boldsymbol\delta,\epsilon_0}(w)
\ge
q_{\min,\alpha}^{\CE,\boldsymbol\delta',\epsilon_0'}(w),
\]
with the analogous relationship holding under PE. By Theorem~\ref{thm:capacity}, these inequalities also imply that robust profit is nonincreasing in the error tolerances. When exact DI is feasible, both CE and PE recover exact DI when all historical and current error tolerances are zero.

\paragraph{Accuracy of the computed supplier policy.}
The demand approximation in Lemma~\ref{lem:e4-grid-error} also provides a bound on the profit of the computed policy. For each $w\in\calW$, let $\widehat q(w)$ denote the estimated worst-case demand, satisfying $0\le\widehat q(w)-q_{\min,\alpha}^M(w)\le e_w$. Define the computed supplier policy and its estimated value by
\[
\widehat w\in\arg\max_{w\in\calW}(w-c)\widehat q(w),
\qquad
\widehat z=\widehat q(\widehat w),
\qquad
\widehat V_\alpha^M
=(\widehat w-c)\widehat q(\widehat w).
\]

\begin{corollary}[Finite-accuracy supplier policy]
\label{cor:e4-policy-error}

If $\max_{w\in\calW}we_w\le\sigma$, then
\begin{equation}
\widehat V_\alpha^M-\sigma
\le
\pi\bigl(\widehat w,\widehat z,
q_{\min,\alpha}^M(\widehat w)\bigr)
\le
V_\alpha^M
\le
\widehat V_\alpha^M.
\label{eq:e4-policy-error}
\end{equation}
\end{corollary}

Since supplier profit is increasing in demand, the middle term in \eqref{eq:e4-policy-error} is the worst-case profit of the computed policy. Thus, $\widehat V_\alpha^M-\sigma$ provides a guaranteed profit level. The required accuracy is achieved by applying Algorithm~\ref{alg:error-global-search} at each $w\in\calW$ with $e_w=\sigma/w$.

\section{Numerical Study}
\label{sec:numerical}

In this section, we use numerical experiments to evaluate the operational value of historical retailer pricing decisions under sparse transaction data. We first compare Q and DI across demand environments, history sizes, and shape restrictions, and then examine the effects of historical observation locations and regular and emergency supply costs. We further assess the robustness of decision information to CE and PE in both historical decisions and current retailer responses. Within each comparison, the methods face the same demand environment, wholesale-price menu, and operating conditions; their difference is the historical decision information they retain.

\subsection{Experimental Setup}

We consider six data-generating distributions (DGPs) with different support and tail behavior: Uniform $[10,20]$, shifted Exponential $(5,0.35)$, shifted Gamma $(8,4,1.3)$, shifted $\chi^2_5$ with location $7$, truncated Normal $(15,2.5^2)$ on $[10,20]$, and $10+10\,\operatorname{Beta}(2,5)$. Each distribution is $\alpha$-regular for every $\alpha\in[0,1]$. Hence, varying $\alpha$ changes the maintained shape restriction without changing the underlying demand environment.

To generate sparse histories, we use

\vspace{-0.3cm}

\[
\begin{aligned}
\calQ_2&=\{1/8,7/8\}, &\calQ_3&=\calQ_2\cup\{4/8\},\\
\calQ_5&=\calQ_3\cup\{2/8,6/8\}, &\calQ_7&=\calQ_5\cup\{3/8,5/8\}.
\end{aligned}
\]
For each $q_i\in\calQ_T$, the population transaction is $s_i=\barF^{-1}(q_i), w_i=\phi_F(s_i)$. Thus, $\calH^{(2)}\subset\calH^{(3)}\subset\calH^{(5)}\subset\calH^{(7)}$: larger histories fill gaps within the same observed range of the demand curve. The associated $w_i$ rationalize the retailer's historical prices under the DGP.

The supplier uses the history-dependent candidate menu $\calW_0=\operatorname{linspace}(2.5,17,50)\cup\{w_i:i=1,\ldots,T\}$, with duplicate prices removed. We set $c=0.2$ and $b=6$ unless otherwise stated, and use $\calW=\{w\in\calW_0:w>6\}$. For each history, all methods and the full-information benchmark use the same menu. Q uses only $(s_i,q_i)$ to construct its ambiguity set, whereas DI also uses the associated $w_i$. For each method $M\in\{Q,DI\}$, let $(w_M,z_M)$ denote the policy selected from the historical data by Theorem~\ref{thm:capacity}. We evaluate this policy under the generating distribution $F$, using $q_F(w)=\bar F(r_F(w))$, and report
\begin{equation}
 R_M^{\mathrm{OOS}}
 =
 \frac{\pi\bigl(w_M,z_M,q_F(w_M)\bigr)}
 {\displaystyle\max_{w\in\calW}(w-c)q_F(w)}.
 \label{eq:oracle-ratio}
\end{equation}
The ratio compares the profit of the selected policy with the full-information supplier's optimal profit over the same menu. Demand for policy evaluation is computed directly from $F$ after the supplier's decision has been made.

\subsection{Value of Historical Decision Information}

\begin{table}[htbp]
\centering
\caption{Q and DI out-of-sample profit-to-oracle ratios}
\label{tab:distribution-main}
\scriptsize
\resizebox{0.8\textwidth}{!}{%
\begin{tabular}{clccccc}
\toprule
$T$ & DGP & $\alpha=0$ & $0.25$ & $0.50$ & $0.75$ & $1$ \\
\midrule
2 & Uniform & 0.503/0.530 & 0.622/0.645 & 0.719/0.727 & 0.788/0.796 & 0.851/\textbf{0.854} \\
 & Exponential & 0.041/0.654 & 0.352/0.661 & 0.602/0.661 & 0.814/0.814 & \textbf{1.000}/\textbf{1.000} \\
 & Gamma & 0.489/0.562 & 0.585/0.654 & 0.666/0.721 & 0.730/0.782 & 0.776/\textbf{0.844} \\
 & Chi-squared & 0.387/0.464 & 0.483/0.565 & 0.624/0.665 & 0.715/0.754 & 0.782/\textbf{0.839} \\
 & Trunc. normal & 0.601/0.639 & 0.669/0.706 & 0.729/0.753 & 0.778/0.805 & 0.812/\textbf{0.843} \\
 & Scaled Beta & 0.598/0.908 & 0.646/\textbf{0.920} & 0.701/\textbf{0.920} & 0.735/\textbf{0.920} & 0.763/\textbf{0.920} \\
\addlinespace
3 & Uniform & 0.611/0.994 & 0.688/\textbf{1.000} & 0.758/\textbf{1.000} & 0.818/\textbf{1.000} & 0.867/\textbf{1.000} \\
 & Exponential & 0.041/0.654 & 0.568/0.661 & 0.698/0.766 & 0.853/0.889 & \textbf{1.000}/\textbf{1.000} \\
 & Gamma & 0.740/0.919 & 0.782/\textbf{0.921} & 0.820/\textbf{0.921} & 0.847/\textbf{0.921} & 0.882/\textbf{0.921} \\
 & Chi-squared & 0.559/0.955 & 0.637/\textbf{0.971} & 0.771/\textbf{0.971} & 0.824/\textbf{0.971} & 0.865/\textbf{0.971} \\
 & Trunc. normal & 0.756/\textbf{0.933} & 0.790/\textbf{0.933} & 0.821/\textbf{0.933} & 0.850/\textbf{0.933} & 0.876/\textbf{0.933} \\
 & Scaled Beta & 0.809/0.908 & 0.827/\textbf{0.920} & 0.852/\textbf{0.920} & 0.871/\textbf{0.920} & 0.892/\textbf{0.920} \\
\addlinespace
5 & Uniform & 0.768/0.994 & 0.820/\textbf{1.000} & 0.853/\textbf{1.000} & 0.883/\textbf{1.000} & 0.911/\textbf{1.000} \\
 & Exponential & 0.607/0.982 & 0.642/0.991 & 0.758/0.991 & 0.902/0.991 & \textbf{1.000}/\textbf{1.000} \\
 & Gamma & 0.850/0.941 & 0.873/\textbf{0.958} & 0.890/\textbf{0.958} & 0.908/\textbf{0.958} & 0.926/\textbf{0.958} \\
 & Chi-squared & 0.793/0.955 & 0.834/\textbf{0.971} & 0.869/\textbf{0.971} & 0.897/\textbf{0.971} & 0.925/\textbf{0.971} \\
 & Trunc. normal & 0.875/\textbf{0.933} & 0.892/\textbf{0.933} & 0.904/\textbf{0.933} & 0.919/\textbf{0.933} & 0.931/\textbf{0.933} \\
 & Scaled Beta & 0.915/0.995 & 0.919/\textbf{1.000} & 0.939/\textbf{1.000} & 0.944/\textbf{1.000} & 0.949/\textbf{1.000} \\
\addlinespace
7 & Uniform & 0.906/0.994 & 0.916/\textbf{1.000} & 0.933/\textbf{1.000} & 0.942/\textbf{1.000} & 0.958/\textbf{1.000} \\
 & Exponential & 0.607/0.982 & 0.651/0.991 & 0.846/0.991 & 0.942/0.991 & \textbf{1.000}/\textbf{1.000} \\
 & Gamma & 0.922/0.973 & 0.930/\textbf{0.996} & 0.943/\textbf{0.996} & 0.953/\textbf{0.996} & 0.961/\textbf{0.996} \\
 & Chi-squared & 0.903/0.994 & 0.917/\textbf{0.997} & 0.929/\textbf{0.997} & 0.944/\textbf{0.997} & 0.959/\textbf{0.997} \\
 & Trunc. normal & 0.916/0.992 & 0.931/\textbf{1.000} & 0.937/\textbf{1.000} & 0.942/\textbf{1.000} & 0.955/\textbf{1.000} \\
 & Scaled Beta & 0.918/0.995 & 0.923/\textbf{1.000} & 0.940/\textbf{1.000} & 0.946/\textbf{1.000} & 0.951/\textbf{1.000} \\
\bottomrule
\end{tabular}}
\begin{minipage}{0.8\textwidth}
\scriptsize
Entries report Q/DI ratios. All ratios are reported to three decimal places.
\end{minipage}
\end{table}

Table~\ref{tab:distribution-main} shows that historical decision information is particularly valuable in the small-data regime. With only two tail observations and $\alpha=0$, DI achieves $46.4\%$--$90.8\%$ of the full-information profit, compared with $4.1\%$--$60.1\%$ under Q. The difference is especially pronounced for exponential demand, where the Q/DI ratio is $0.041/0.654$. Adding the median observation raises the uniform DI ratio from $0.530$ to $0.994$ and the Gamma ratio from $0.562$ to $0.919$, illustrating the value of filling a central gap without expanding the observed range. With three transactions and $\alpha=0$, DI attains $65.4\%$--$99.4\%$ of the full-information profit and outperforms Q in all six environments. Moreover, DI with three observations outperforms Q with seven observations in four environments. Thus, the optimality information in a few historical decisions can be more useful than several additional price--demand observations.

Increasing the history size does not necessarily produce a strict improvement. Additional observations restrict the feasible demand curves and may change the selected wholesale price and stock, but their value depends on whether these changes are decision-relevant. For scaled-Beta demand at $\alpha=0$, the DI ratio rises from $0.908$ with three observations to $0.995$ with five, whereas further observations produce little additional gain. The shape restriction also affects the value of historical decisions. With two uniform observations, increasing $\alpha$ from zero to one changes the Q/DI ratios from $0.503/0.530$ to $0.851/0.854$, leaving a much smaller incremental gain from DI. Hence, the benefit of a transaction depends on both its location and the demand uncertainty already ruled out by the shape restriction.

\begin{figure}[H]
\centering
\includegraphics[width=0.8\textwidth]{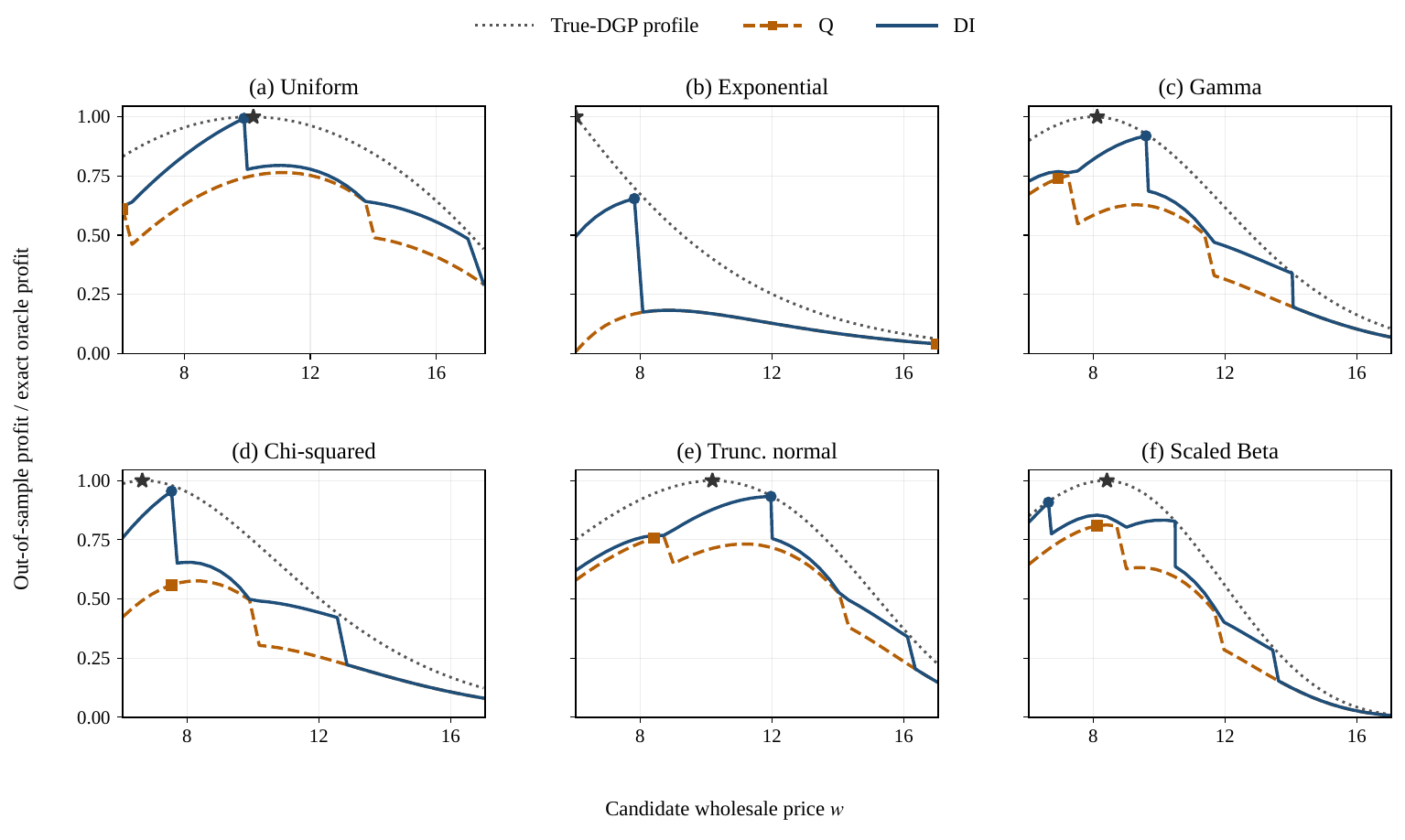}
\caption{Q and DI profit profiles.}
\label{fig:pointwise-profit}
\end{figure}

Figure~\ref{fig:pointwise-profit} illustrates how historical decision information affects the supplier's out-of-sample profit profile and policy choice. For $T=3$, $\alpha=0$, $c=0.2$, and $b=6$, DI outperforms Q at every candidate wholesale price, demonstrating the value of incorporating historical pricing decisions. DI selects a different wholesale price from Q in five of the six demand environments. Notably, under chi-squared demand, the two methods select the same wholesale price, but DI increases regular stock from $0.236$ to $0.500$, raising the profit ratio from $0.559$ to $0.955$. Under Gamma demand, DI selects a different wholesale price from Q, improving the profit ratio from $0.740$ to $0.919$. These results show that historical decision information can improve supplier performance through both more effective contract selection and better stocking decisions.

\subsection{Location of Historical Transactions}

The number of observations alone does not determine their value; where those observations lie on the demand curve can be equally important. To isolate this location effect, we fix $T=3$ and compare four prespecified, nonnested histories:

\vspace{-0.3cm}

\begin{equation}
\begin{aligned}
\calQ_{\mathrm{low}}&=\{1/8,2/8,3/8\}, &
\calQ_{\mathrm{center}}&=\{3/8,4/8,5/8\},\\
\calQ_{\mathrm{high}}&=\{5/8,6/8,7/8\}, &
\calQ_{\mathrm{spread}}&=\{1/8,4/8,7/8\}.
\end{aligned}
\label{eq:history-placement}
\end{equation}
The first three designs shift the observation window across the demand curve, while the spread design covers a wider range with the same number of records. All comparisons use $\alpha=0$ and $c=0.2$, with $b=6$.

\begin{figure}[H]
\centering
\includegraphics[width=\textwidth]{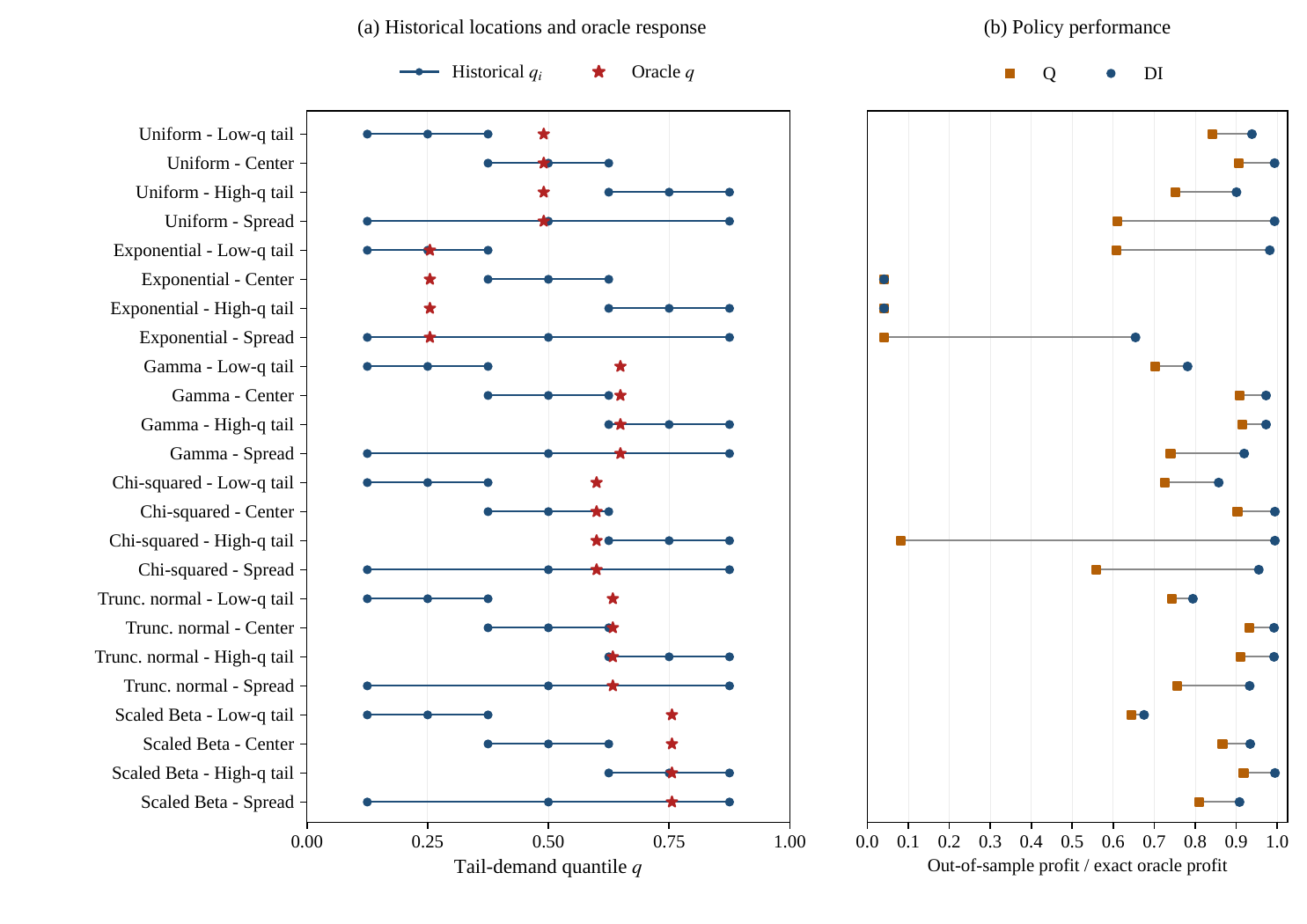}
\caption{Effect of historical information placement.}
\label{fig:history-placement}
\end{figure}

Figure~\ref{fig:history-placement} shows that the value of a sparse transaction history depends strongly on where its observations lie relative to the demand region relevant to the supplier's decision. Panel (a) compares the three historical demand quantiles with the demand induced by the full-information menu choice. Histories concentrated in the low-$q$ tail perform less well for Gamma, truncated-normal, and scaled-Beta demand. For scaled-Beta demand, moving the observations from the low-$q$ to the high-$q$ tail raises the Q/DI ratio from $0.644/0.675$ to $0.918/0.995$. Similar improvements from shifting the observations toward the decision-relevant region appear under Gamma and truncated-normal demand. Thus, the location of a few transactions can matter as much as their number.

Panel (b) further shows that DI generally reduces the sensitivity of performance to the placement of historical observations. DI outperforms Q in 22 of the 24 DGP--design comparisons. The difference is particularly pronounced for the chi-squared high-$q$ design, where the Q/DI ratios are $0.081/0.994$, highlighting the information contained in historical wholesale prices beyond the observed price--demand pairs. A similar pattern emerges under uniform demand: the centered history yields a Q/DI ratio of $0.906/0.994$, compared with $0.611/0.994$ for the spread history. While the larger interpolation gaps substantially weaken Q, DI achieves nearly identical performance across the two designs. These results suggest that historical decision information can mitigate the limitations of sparse or poorly positioned observations.

\subsection{Supply Costs and Information Value}

We next examine how the value of historical decision information changes with supply costs. Holding $T=3$ and $\alpha=0$ fixed, we first vary $c\in\{0.2,1,4\}$ with $b=6$, and then vary $b\in\{3,4.5,6\}$ with $c=0.2$. 

Table~\ref{tab:cost-distribution} shows that DI outperforms Q at all reported regular-supply costs, although the magnitude of its advantage varies across demand environments. As $c$ increases from $0.2$ to $4$, the DI advantage grows from $61.3$ to $80.0$ percentage points under exponential demand, whereas it narrows from $38.3$ to $5.0$ points under uniform demand. The advantage also increases under Gamma and truncated-normal demand. Our numerical experiments further reveal that changes in regular-supply costs can alter the selected wholesale price and stocking level. Thus, regular-supply costs affect both the supplier's decisions and the relative profit benefit of historical decision information.

\begin{table}[H]
\centering
\caption{Effect of regular and emergency supply costs}
\label{tab:cost-distribution}
\small
\setlength{\tabcolsep}{4pt}
\resizebox{0.8\textwidth}{!}{%
\begin{tabular}{lcccccc}
\toprule
& \multicolumn{3}{c}{Regular-supply cost ($b=6$)} & \multicolumn{3}{c}{Emergency-sourcing cost ($c=0.2$)} \\
\cmidrule(lr){2-4}\cmidrule(lr){5-7}
DGP & $c=0.2$ & $c=1$ & $c=4$ & $b=3$ & $b=4.5$ & $b=6$ \\
\midrule
Uniform & 0.611/0.994 & 0.776/0.990 & 0.878/0.928 & 0.726/0.997 & 0.668/0.995 & 0.611/0.994 \\
Exponential & 0.041/0.654 & 0.047/0.680 & 0.111/0.911 & 0.052/0.678 & 0.046/0.666 & 0.041/0.654 \\
Gamma & 0.740/0.919 & 0.747/0.936 & 0.671/0.991 & 0.858/0.924 & 0.799/0.922 & 0.740/0.919 \\
Chi-squared & 0.559/0.955 & 0.589/0.968 & 0.775/0.938 & 0.777/0.968 & 0.668/0.961 & 0.559/0.955 \\
Trunc. normal & 0.756/0.933 & 0.764/0.944 & 0.725/0.986 & 0.854/0.934 & 0.805/0.934 & 0.756/0.933 \\
Scaled Beta & 0.809/0.908 & 0.814/0.881 & 0.865/0.929 & 0.907/0.910 & 0.858/0.909 & 0.809/0.908 \\
\bottomrule
\end{tabular}}
\end{table}

Emergency-sourcing costs have a different effect. Our numerical experiments show that varying $b$ leaves the robust
pricing and stocking decisions unchanged, consistent with Theorem~\ref{thm:capacity}. However, higher emergency-sourcing costs reduce the profit earned from demand exceeding regular stock. For chi-squared demand, increasing $b$ from $3$ to $6$ widens the DI advantage from $19.1$ to $39.6$ percentage points. Similar patterns appear in most other demand environments, suggesting that more expensive emergency sourcing generally increases the relative profit benefit of historical decision information.

\subsection{Robustness to Condition Error}

We next examine robustness to condition error in both the historical decisions and the retailer's response to a new wholesale price. We set $T=5$, $\alpha=0$, and $c=0.2$, with $b=6$. In each replication, the observed historical wholesale prices are perturbed according to

\vspace{-0.8cm}

\begin{align}
 w_i^{\mathrm{obs}}(\delta)
 =\phi_F(s_i)-\delta &\xi_i,
 &&\xi_i\sim\operatorname{Unif}[-1,1],
 \label{eq:error-experiment}\\
 |\phi_F(s)-w|\leq&\delta
 &&\text{for the response to each candidate }w.
 \label{eq:current-error-experiment}
\end{align}

We consider $\delta\in\{0,0.25,0.50,0.75,1.00\}$, with each $\xi_i$ held fixed across all values of $\delta$. We compare Q, which discards historical wholesale-price information, with Error-aware DI, which uses the historical CE bands. Both methods face the same current CE and the same menu in each replication. We use 50 replications for each DGP. After each method has selected its policy, we evaluate demand as $\bar F(r_F(w_M+\delta))$, retaining the least favorable current CE realization in Lemma~\ref{lem:current-target-monotonicity}. The supplier receives $w_M$ per unit, and its profit is evaluated by \eqref{eq:oracle-ratio}. Both error figures report mean ratios with bands of one sample standard deviation. The oracle uses the same menu and an exact retailer response.

\begin{figure}[H]
\centering
\includegraphics[width=0.8\textwidth]{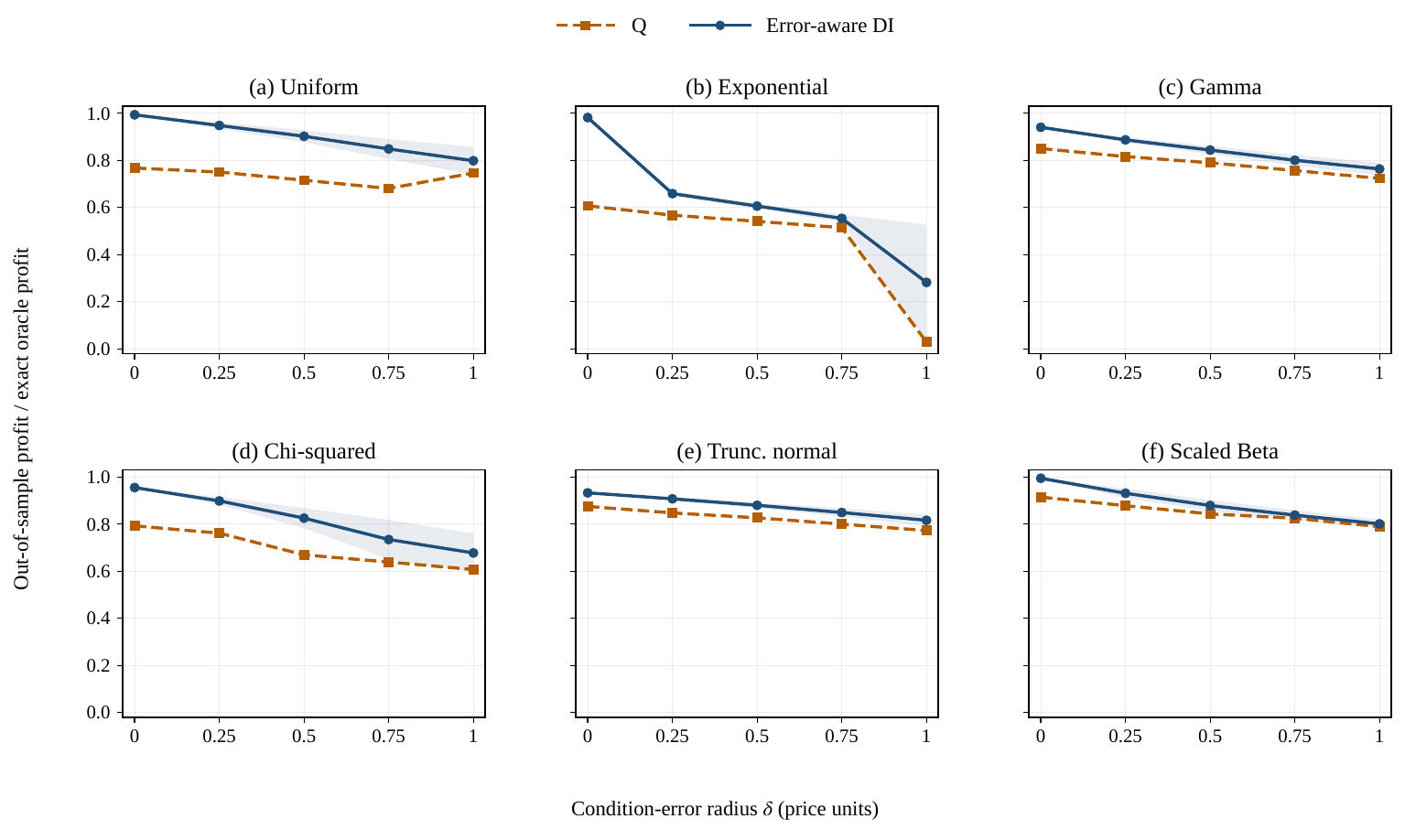}
\caption{Robustness to condition error.}
\label{fig:error-performance}
\end{figure}

Figure~\ref{fig:error-performance} shows that larger CE generally reduces out-of-sample performance. This decline has two sources: a larger current tolerance allows more adverse retailer responses to the new contract, while a larger historical tolerance weakens the information extracted from past pricing decisions. Under uniform demand, the DI ratio falls from $0.994$ without error to $0.798$ at $\delta=1$; the corresponding exponential ratio falls from $0.982$ to $0.282$. Despite this deterioration, Error-aware DI outperforms Q in every reported mean comparison. At $\delta=0.25$, its advantage ranges from $5.3$ to $19.8$ percentage points, and a positive gap of $1.2$--$25.3$ points remains at $\delta=1$. Thus, the value of historical decision information generally declines with error but does not disappear once exact optimality is relaxed.

\subsection{Robustness to Price Error}
\label{subsec:price-error-experiment}

Finally, we examine robustness to price error using the same baseline values $T=5$, $\alpha=0$, and $c=0.2$, with $b=6$. In replication $j$, we generate

\vspace{-0.8cm}

\begin{align}
\widetilde s_{ij}(\eta)
&=s_i-\eta\zeta_{ij},
&
w_{ij}^{\mathrm{obs}}(\eta)
&=\phi_F\!\left(\widetilde s_{ij}(\eta)\right),
\label{eq:price-error-experiment}
\end{align}
where $\zeta_{ij}$ is sampled uniformly from the feasible portion of $[-1,1]$. The observed demand remains $q_i=\bar F(s_i)$. We set the historical and current PE radii equal, $\eta_0=\eta$. For each DGP, the maximum radius is chosen as $20\%$ of the historical retail-price range, rounded down to a multiple of $0.04$. The resulting maximum radii are $(1.48,1.08,1.12,1.36,1.04,0.72)$. We use five equally spaced radii and 50 replications for each DGP. The perturbation intervals are intersected with the feasible pricing region to ensure valid historical transactions.

\begin{figure}[H]
\centering
\includegraphics[width=0.8\textwidth]{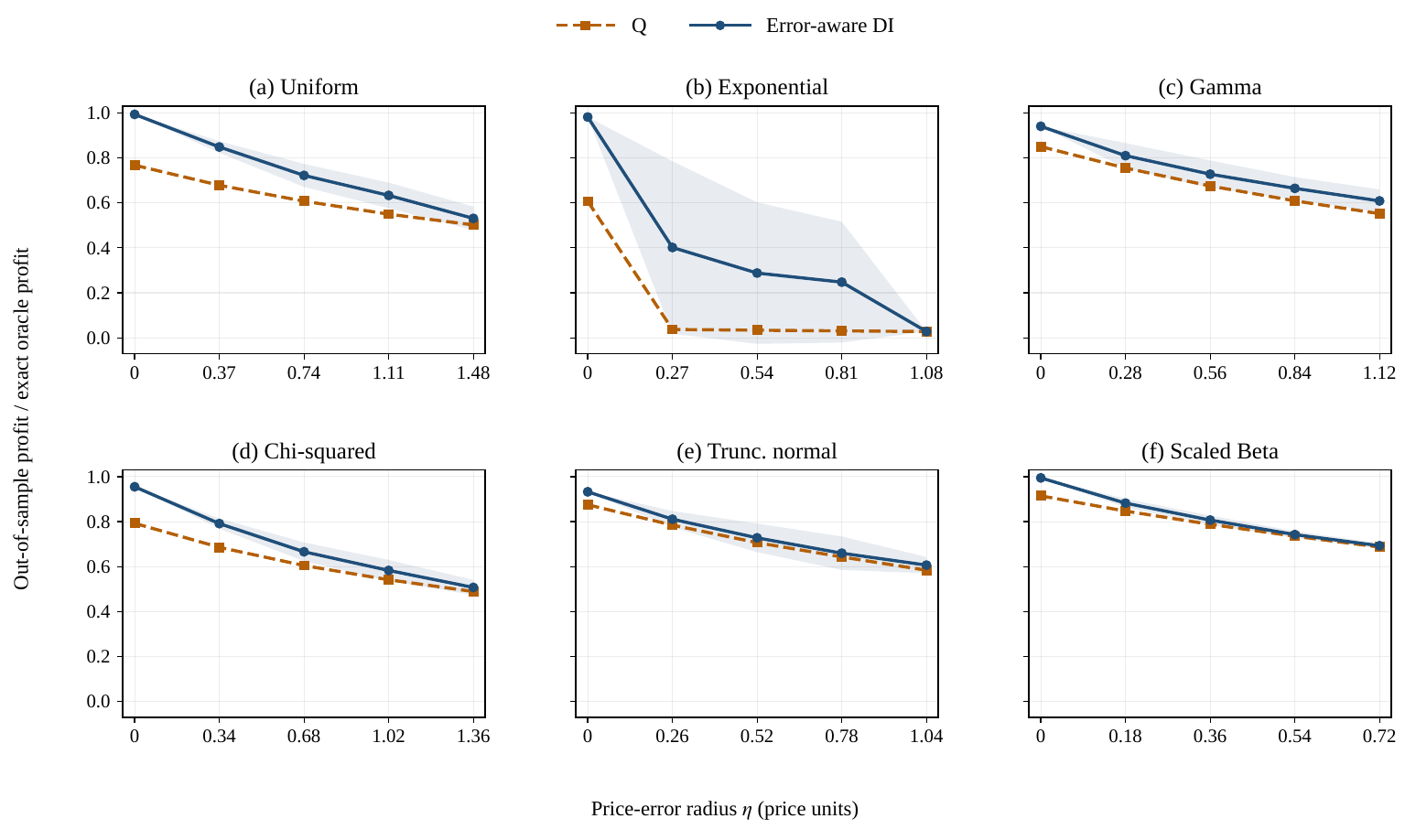}
\caption{Robustness to price error. Each panel uses its indicated absolute-radius grid.}
\label{fig:price-error-performance}
\end{figure}

We again compare Q with Error-aware DI, now using the historical PE neighborhoods for DI. Both methods face the same current price error. For each selected policy, we evaluate demand at the implemented price $r_F(w_M)+\eta$ under the generating distribution. The profit ratio and exact-response oracle are the same as in the CE experiment.

Figure~\ref{fig:price-error-performance} shows that larger PE generally reduces out-of-sample profit, as a wider current price deviation leads to less favorable demand and wider historical neighborhoods weaken decision information. Error-aware DI nevertheless improves on Q at the first positive radius in all six demand environments, with gains of $2.6$--$36.5$ percentage points. At the largest radius, DI remains above Q in five environments, including gains of $5.7$ points for Gamma, $2.9$ for uniform, and $2.3$ for truncated-normal demand. Thus, historical pricing decisions remain informative even when the observed retail prices are only approximately optimal, although their value tends to decrease as the permitted price neighborhoods widen.

\section{Conclusion}
\label{sec:conclusion}

This paper investigates how suppliers can leverage sparse historical transactions to make pricing and stocking decisions when retailers possess superior information about end-market demand. Beyond the observed price--demand pairs, historical wholesale prices provide additional information about the retailer's pricing decisions. By exploiting this decision context through inverse optimization, we construct a decision-informed ambiguity set that refines the set of plausible demand distributions.

We show that DI weakly improves the supplier's worst-case demand guarantee and robust profit relative to the quantile-only benchmark, and derive the corresponding robust pricing and stocking policy. Numerical experiments demonstrate that historical decision information can generate substantial profit improvements even with only a few transactions. The magnitude of these gains depends on the demand environment and the placement of historical observations.

We further extend the framework to imperfect retailer decisions through condition error and price error, which capture deviations in the retailer's optimality condition and pricing decision, respectively. We characterize the resulting ambiguity sets, establish when historical decision information becomes redundant, and develop a unified approach for computing worst-case demand. Numerical results show that decision information remains valuable under moderate errors, although its benefits generally diminish as error tolerances increase.

Future research could incorporate sampling uncertainty in observed demand, nonstationary demand environments, richer supplier-side operations and contracting structures, or alternative models of imperfect retailer behavior.

\bibliographystyle{apalike}
\setlength{\bibsep}{2pt}
\bibliography{references}

@article{AkanAtaLariviere2011,
  author={Akan, Mustafa and Ata, Baris and Lariviere, Martin A.},
  title={Asymmetric Information and Economies of Scale in Service Contracting},
  journal={Manufacturing \& Service Operations Management},
  year={2011}, volume={13}, number={1}, pages={58--72}
}

@article{AllouahBahamouBesbes2022,
  author={Allouah, Amine and Bahamou, Abdellah and Besbes, Omar},
  title={Pricing with Samples},
  journal={Operations Research}, year={2022}, volume={70}, number={2}, pages={1088--1104}
}

@article{BergemannSchlag2011,
  author={Bergemann, Dirk and Schlag, Karl H.},
  title={Robust Monopoly Pricing},
  journal={Journal of Economic Theory}, year={2011}, volume={146}, number={6}, pages={2527--2543}
}

@article{BesbesZeevi2009,
  author={Besbes, Omar and Zeevi, Assaf},
  title={Dynamic Pricing without Knowing the Demand Function: Risk Bounds and Near-Optimal Algorithms},
  journal={Operations Research}, year={2009}, volume={57}, number={6}, pages={1407--1420}
}

@techreport{CMA2024LoyaltyPricing,
  author={{Competition and Markets Authority}},
  title={Review of Loyalty Pricing in the Groceries Sector: Findings Report},
  institution={Competition and Markets Authority}, address={London}, year={2024}
}

@article{CheungSimchiLeviWang2017,
  author={Cheung, Wang Chi and Simchi-Levi, David and Wang, He},
  title={Dynamic Pricing and Demand Learning with Limited Price Experimentation},
  journal={Operations Research}, year={2017}, volume={65}, number={6}, pages={1722--1731}
}

@article{CorbettZhouTang2004,
  author={Corbett, Charles J. and Zhou, Deming and Tang, Christopher S.},
  title={Designing Supply Contracts: Contract Type and Information Asymmetry},
  journal={Management Science}, year={2004}, volume={50}, number={4}, pages={550--559}
}

@article{DelageYe2010,
  author={Delage, Erick and Ye, Yinyu},
  title={Distributionally Robust Optimization under Moment Uncertainty with Application to Data-Driven Problems},
  journal={Operations Research}, year={2010}, volume={58}, number={3}, pages={595--612}
}

@article{GaoKleywegt2023,
  author={Gao, Rui and Kleywegt, Anton J.},
  title={Distributionally Robust Stochastic Optimization with Wasserstein Distance},
  journal={Mathematics of Operations Research}, year={2023}, volume={48}, number={2}, pages={603--655}
}

@techreport{GeHeWangWang2025,
  author={Ge, Yan and He, Simai and Wang, Zhen and Wang, Zizhuo},
  title={Optimal Robust Pricing with Minimal Information},
  institution={SSRN}, year={2025}
}

@article{HuangYuZhao2026,
  author={Huang, Guohua and Yu, Guodong and Zhao, Xuejun},
  title={Refining Data-Driven Upfront Reservation Discount Pricing via Inverse Inferring Newsvendor Transactions},
  journal={Production and Operations Management}, year={2026}, volume={35}, number={5}, pages={1766--1784}
}

@article{KalkanciChenErhun2011,
  author={Kalkanci, Basak and Chen, Kay-Yut and Erhun, Feryal},
  title={Contract Complexity and Performance under Asymmetric Demand Information: An Experimental Evaluation},
  journal={Management Science}, year={2011}, volume={57}, number={4}, pages={689--704}
}

@article{LevyEtAl1997,
  author={Levy, Daniel and Bergen, Mark and Dutta, Shantanu and Venable, Robert},
  title={The Magnitude of Menu Costs: Direct Evidence from Large U.S. Supermarket Chains},
  journal={Quarterly Journal of Economics}, year={1997}, volume={112}, number={3}, pages={791--825}
}

@article{MohajerinEsfahaniKuhn2018,
  author={Mohajerin Esfahani, Peyman and Kuhn, Daniel},
  title={Data-Driven Distributionally Robust Optimization Using the Wasserstein Metric: Performance Guarantees and Tractable Reformulations},
  journal={Mathematical Programming}, year={2018}, volume={171}, number={1--2}, pages={115--166}
}

@article{NambiarSimchiLeviWang2019,
  author={Nambiar, Mila and Simchi-Levi, David and Wang, He},
  title={Dynamic Learning and Pricing with Model Misspecification},
  journal={Management Science}, year={2019}, volume={65}, number={11}, pages={4980--5000}
}

@article{NasserTurcic2019,
  author={Nasser, Sherif and Turcic, Danko},
  title={Temporary Contract Adjustment to a Retailer with a Private Demand Forecast},
  journal={Management Science}, year={2019}, volume={65}, number={1}, pages={209--229}
}

@article{YuDongSun2025,
  author={Yu, Guodong and Dong, Pengcheng and Sun, Huiping},
  title={Refined Wasserstein Distributionally Robust Optimization for Contract Pricing: The Value of Optimality Conditions in Transactions},
  journal={INFORMS Journal on Computing}, year={2025}, volume={38}, number={2}, pages={397--413}
}

@article{ZbarackiEtAl2004,
  author={Zbaracki, Mark J. and Ritson, Mark and Levy, Daniel and Dutta, Shantanu and Bergen, Mark},
  title={Managerial and Customer Costs of Price Adjustment: Direct Evidence from Industrial Markets},
  journal={Review of Economics and Statistics}, year={2004}, volume={86}, number={2}, pages={514--533}
}

@article{ZhaoHaskellYu2024,
  author={Zhao, Xuejun and Haskell, William B. and Yu, Guodong},
  title={Supply Chain Contracts in the Small Data Regime},
  journal={Manufacturing \& Service Operations Management}, year={2024}, volume={26}, number={4}, pages={1387--1401}
}

@article{denBoer2015,
  author  = {Arnoud V. den Boer},
  title   = {Dynamic Pricing and Learning: Historical Origins, Current Research, and New Directions},
  journal = {Surveys in Operations Research and Management Science},
  volume  = {20},
  number  = {1},
  pages   = {1--18},
  year    = {2015},
  doi     = {10.1016/j.sorms.2015.03.001}
}

@article{myerson1981optimal,
  author  = {Myerson, Roger B.},
  title   = {Optimal Auction Design},
  journal = {Mathematics of Operations Research},
  volume  = {6},
  number  = {1},
  pages   = {58--73},
  year    = {1981},
  doi     = {10.1287/moor.6.1.58}
}

@inproceedings{cole2014sample,
  author    = {Cole, Richard and Roughgarden, Tim},
  title     = {The Sample Complexity of Revenue Maximization},
  booktitle = {Proceedings of the 46th Annual ACM Symposium on Theory of Computing},
  pages     = {243--252},
  year      = {2014},
  doi       = {10.1145/2591796.2591867}
}

@article{bagnoli2005logconcave,
  author  = {Bagnoli, Mark and Bergstrom, Ted},
  title   = {Log-Concave Probability and Its Applications},
  journal = {Economic Theory},
  volume  = {26},
  number  = {2},
  pages   = {445--469},
  year    = {2005},
  doi     = {10.1007/s00199-004-0514-4}
}

@article{allouah2023optimal,
  author  = {Allouah, Amine and Bahamou, Achraf and Besbes, Omar},
  title   = {Optimal Pricing with a Single Point},
  journal = {Management Science},
  volume  = {69},
  number  = {10},
  pages   = {5866--5882},
  year    = {2023},
  doi     = {10.1287/mnsc.2023.4683}
}

@article{chen2026joint,
  title={Joint planning and operations of wind power under decision-dependent uncertainty},
  author={Chen, Zhiqiang and Xu, Wei and Chen, Caihua and Cui, Jingshi and Hu, Qian},
  journal={Production and Operations Management},
  year={2026},
  publisher={SAGE Publications Sage CA: Los Angeles, CA}
}

\clearpage
\appendix

\section{Proofs}
\label{app:proofs}

\vspace{-0.3cm}

\paragraph{Notation.}
Throughout the appendices, let $\rho=1-\alpha$ and $u_i=\bar\Gamma_\alpha^{-1}(q_i)$ for each observed demand level. For an admissible distribution $F$, write $U(t)=\bar\Gamma_\alpha^{-1}(\bar F(t))$ and $g(t)=U'(t)$ on the support interior. On any ordered price grid $t_1<\cdots<t_m$, $U_j=U(t_j)$, $g_j=g(t_j)$, and $\Delta_j=t_{j+1}-t_j$ denote the nodal values, slopes, and spacings. For $M\in\{\CE,\PE\}$, $\theta_M$ and $d_M$ are the effective pricing target and implementation adjustment in \eqref{eq:error-current-parameters}; $r$ denotes a candidate optimal retail price and $y=r+d_M$ its implemented price. We use $\mathcal R_M$ for the feasible response prices and $\mathcal J_M(r)$ for grid intervals required to be affine. The superscript in $x^\circ$ labels a nodal vector satisfying the strict interpolation test.

\vspace{-0.4cm}

\subsection{Proof of Proposition~\ref{prop:history-feasible}}
\label{app:proof-history-feasible}

\vspace{-0.3cm}

The generalized-Pareto identities
$
\bar\Gamma_\alpha'(u)=-\bar\Gamma_\alpha(u)^{1+\rho},
\bar\Gamma_\alpha(u)^{-\rho}=1+\rho u
$
imply

\vspace{-0.3cm}

\begin{equation}
g(s)=f(s)\bar F(s)^{-(1+\rho)}>0,
\qquad
\phi_F(s)=s-\frac{1+\rho U(s)}{g(s)}.
\label{eq:app-phi-transformed}
\end{equation}
Under Assumption~\ref{assump:q-feasible}, $g$ is locally absolutely continuous, and differentiation gives

\vspace{-0.3cm}

\begin{equation}
\phi_F'(s)-\alpha
=\frac{[1+\rho U(s)]g'(s)}{g(s)^2}
\qquad\text{a.e.}
\label{eq:app-phi-derivative}
\end{equation}

Hence $\alpha$-regularity is equivalent to convexity of $U$. Extending $U$ by $0$ where $\bar F=1$ and by $+\infty$ where $\bar F=0$ allows us to use the same bounds on the entire price range.

Set $(s_0,q_0)=(0,1)$. The line through two adjacent transformed observations is

\vspace{-0.3cm}

\[
A_i(v)=u_i+\frac{u_{i+1}-u_i}{s_{i+1}-s_i}(v-s_i),
\qquad i=0,\ldots,T-1.
\]

By convexity, $U(v)\le A_i(v)$ between the observations and $U(v)\ge\max\{0,A_i(v)\}$ outside them. Since
$
G_{\alpha,i}(v)=\bar\Gamma_\alpha\!\left(\max\{0,A_i(v)\}\right)
$
and $\bar\Gamma_\alpha$ is decreasing, the two inequalities give the interpolation lower bound and the extension upper bound, respectively. Their combination proves \eqref{eq:q-bound-validity}.

For $F\in\mathcal F_\alpha^{DI}(\mathcal H)$, historical optimality and \eqref{eq:app-phi-transformed} give
$
g(s_i)=\frac{1+\rho u_i}{s_i-w_i}.
$
Convexity then gives the supporting-line bound

\vspace{-0.5cm}

\[
U(v)\ge\max\{0,u_i+g(s_i)(v-s_i)\},
\qquad
\bar F(v)\le H_\alpha(v\mid s_i,q_i,w_i).
\]

Combining these bounds with the Q bounds proves \eqref{eq:di-bound-validity}.
\qed

\subsection{Proof of Theorem~\ref{thm:insertion}}
\label{app:proof-response-boundary}

Fix $F$ and $w_0$. Each admissible distribution $F$ induces a set $\mathcal S_F(w_0)$ of optimal retail prices, which may contain one or multiple elements. Each optimal price $r\in\mathcal S_F(w_0)$ generates demand $\bar F(r)$. For fixed $w_0$ and $z$, the supplier's profit $\pi(w_0,z,q)$ is nondecreasing in demand when $w_0\ge b$, with slopes $w_0$ and $w_0-b$. Therefore, the supplier's least favorable retailer response is the one inducing the smallest demand. Since $\bar F(r)$ is nonincreasing in the retail price $r$, this corresponds to the largest optimal retail price.

Assumption~\ref{assump:q-feasible} ensures that the retailer's profit is continuous and vanishes as $s\to\infty$. If its maximum is positive, $\mathcal S_F(w_0)$ is nonempty and compact, giving

\vspace{-0.5cm}

\begin{equation}
r_F(w_0)=\max\mathcal S_F(w_0),
\qquad
q_F(w_0)
=\bar F(r_F(w_0))
=\min_{r\in\mathcal S_F(w_0)}\bar F(r).
\label{eq:app-selection-value-equivalence}
\end{equation}

If the maximum profit is zero, then $\bar F(w_0)=0$; setting $r_F(w_0)=w_0$ gives the same minimum-demand identity.

Now set
$
a=\inf_{F\in\mathcal F_\alpha^M(\mathcal H)}q_F(w_0).
$
By continuity and monotonicity of $\pi$ in demand,

\vspace{-0.5cm}

\[
\inf_{F\in\mathcal F_\alpha^M(\mathcal H)}
\pi(w_0,z,q_F(w_0))
=\pi(w_0,z,a)
=w_0a-cz-b(a-z)^+.
\]

The global response set gives

\vspace{-0.5cm}

\[
a
=\inf_{F\in\mathcal F_\alpha^M(\mathcal H)}
\min_{r\in\mathcal S_F(w_0)}\bar F(r)
=\inf_{(r,q)\in\mathcal C_\alpha^M(w_0)}q.
\]

Together, these identities prove \eqref{eq:positive-worst-profit} and \eqref{eq:worst-demand}.
\qed

\vspace{-0.3cm}

\subsection{Proof of Lemma~\ref{lem:current-target-monotonicity}}
\label{app:current-error}

\vspace{-0.3cm}

Fix $F$ and $\theta_1<\theta_2$. Choose $r_k\in\mathcal S_F(\theta_k;w_0)$ and let $q_k=\bar F(r_k)$ for $k=1,2$. Retailer optimality gives

\vspace{-0.5cm}

\[
(r_1-\theta_1)q_1\ge(r_2-\theta_1)q_2,
\qquad
(r_2-\theta_2)q_2\ge(r_1-\theta_2)q_1.
\]

Adding the inequalities gives $(\theta_2-\theta_1)(q_1-q_2)\ge0$, hence $q_1\ge q_2$. Demand is therefore nonincreasing in the target, with the worst realization at $\theta_0^{\mathrm{wc}}=w_0+\epsilon_0$.
\qed

\vspace{-0.3cm}

\subsection{Proof of Theorem~\ref{thm:decision-space-reduction}}
\label{app:proof-decision-space-reduction}

\vspace{-0.3cm}

Fix $r>\theta_M$ and $y=r+d_M$. Let $\mathfrak F_M(r)$ be the admissible distributions satisfying the historical restrictions for which $r$ is optimal under $\theta_M$. We first consider price nodes in the support interior and then handle bounded support.

\vspace{-0.4cm}

\paragraph{Pricing restrictions.}
For a target $\theta$, retailer profit has derivative
\begin{equation}
\frac{d}{dt}\{(t-\theta)\bar F(t)\}
=f(t)\{\theta-\phi_F(t)\}.
\label{eq:app-K}
\end{equation}

As $\phi_F$ is nondecreasing, interior optimality at $r$ is equivalent to $\phi_F(r)=\theta_M$. Equation~\eqref{eq:app-phi-transformed} then gives
$
(r-\theta_M)g(r)-\rho U(r)=1,
$
which is \eqref{eq:error-lp-current}. The observed demand equalities give $U(s_i)=u_i$. The same identity converts the CE restrictions into

\vspace{-0.5cm}

\[
(s_i-w_i-\delta_i)g(s_i)
\le 1+\rho u_i
\le(s_i-w_i+\delta_i)g(s_i).
\]
For PE, \eqref{eq:pe-endpoint-crossing-main} gives $\phi_F(\underline{s}_i)\le w_i$ when $\underline{s}_i>w_i$, and $\phi_F(\overline{s}_i)\ge w_i$. Multiplying by the positive derivatives gives

\vspace{-0.5cm}

\[
(\underline{s}_i-w_i)g(\underline{s}_i)\le1+\rho U(\underline{s}_i),
\qquad
(\overline{s}_i-w_i)g(\overline{s}_i)\ge1+\rho U(\overline{s}_i).
\]
These are exactly the historical constraints \eqref{eq:error-lp-ce}--\eqref{eq:error-lp-pe-right}.

\vspace{-0.4cm}

\paragraph{Interpolation between price nodes.}
For adjacent nodes, put $D_j=(U_{j+1}-U_j)/\Delta_j$. Convexity requires $g_j\le D_j\le g_{j+1}$. More precisely, a continuous nondecreasing derivative must satisfy

\vspace{-0.5cm}

\begin{equation}
g_j=D_j=g_{j+1}
\qquad\text{or}\qquad
g_j<D_j<g_{j+1}.
\label{eq:app-interpolation-criterion}
\end{equation}

Indeed, $D_j$ is the average derivative over the interval. For a continuous nondecreasing $g$, equality with either endpoint slope forces $g$ to be constant. Otherwise, with $p_j=(g_{j+1}-D_j)/(D_j-g_j)>0$, take
\begin{equation}
U'(t_j+v)
=g_j+(g_{j+1}-g_j)\left(\frac{v}{\Delta_j}\right)^{p_j},
\qquad 0\le v\le\Delta_j.
\label{eq:app-derivative-interpolant}
\end{equation}
This gives a continuous, nondecreasing derivative with average $D_j$, and hence matches the specified values and slopes at both endpoints.

\vspace{-0.4cm}

\paragraph{The interpolation-consistency test.}
Let $P$ be the feasible set of nonnegative nodal vectors satisfying \eqref{eq:error-lp-data}--\eqref{eq:error-lp-pe-right}. For each adjacent pair of nodes, define the interpolation slacks

\vspace{-0.5cm}

\[
L_j=U_{j+1}-U_j-\Delta_jg_j,
\qquad
R_j=\Delta_jg_{j+1}-U_{j+1}+U_j.
\]

Start with $\mathcal J=\varnothing$ and compute the supremum of each slack over $P$. If either supremum is zero, the interval must be affine by \eqref{eq:app-interpolation-criterion}. Impose $L_j=R_j=0$, add $j$ to $\mathcal J$, and repeat. If these equalities make $P$ empty, reject $r$. Otherwise, after at most $m-1$ rounds, the procedure identifies $\mathcal J_M(r)$.

On the resulting nonempty feasible set $P$, solve

\vspace{-0.5cm}

\begin{equation}
\begin{aligned}
\tau^*=\max_{U,g,\tau}\quad &\tau\\
\text{s.t.}\quad &(U,g)\in P,\quad 0\le\tau\le1,\\
&U_1\ge\tau,\quad g_j\ge\tau, &&j=1,\ldots,m,\\
&L_j\ge\tau,\quad R_j\ge\tau, &&j\notin\mathcal J_M(r).
\end{aligned}
\label{eq:app-strict-witness}
\end{equation}

The candidate passes if $\tau^*>0$. The test then supplies a vector $x^\circ$ with positive nodal values and derivatives, and strict interpolation slacks on every unmarked interval. Such a vector can be interpolated using affine segments on $\mathcal J_M(r)$ and \eqref{eq:app-derivative-interpolant} elsewhere. Conversely, every admissible full-support curve satisfies the forced equalities; averaging feasible vectors with positive individual slacks yields a vector with all the required strict inequalities. Thus the test characterizes the interpolable candidates. The LP uses the closed feasible set $P$; we next show that this does not change its supremum.

\vspace{-0.4cm}

\paragraph{Completing the demand curve.}
The support constraint $U_1\le t_1g_1$ follows from convexity and $U(0)=0$. For positive $U_1,g_1$ satisfying this constraint, a left extension is

\vspace{-0.3cm}

\begin{equation}
U(t)=U_1\left(\frac{t}{t_1}\right)^p,
\qquad
p=\frac{t_1g_1}{U_1}\ge1,
\qquad 0\le t\le t_1.
\label{eq:app-left-extension}
\end{equation}
It matches both $U_1$ and $g_1$ at $t_1$ and has positive derivative for $t>0$. At the last node, use

\vspace{-0.5cm}

\begin{equation}
U(t)=U_m+g_m(t-t_m)+\kappa(t-t_m)^2,
\qquad t\ge t_m,\quad\kappa>0.
\label{eq:app-right-extension}
\end{equation}

Together with the nodal interpolants, these extensions produce a convex $U$ with positive, locally absolutely continuous derivative on $(0,\infty)$. The resulting demand curve $\bar F(t)=\bar\Gamma_\alpha(U(t))$ is admissible, and the quadratic tail ensures $t\bar F(t)\to0$.

Now take any feasible vector $\widehat x$ of the LP with weak inequalities. For $0<\lambda<1$,

\vspace{-0.5cm}

\[
x^\lambda=(1-\lambda)\widehat x+\lambda x^\circ
\]
preserves the LP constraints and the affine intervals, while making all other interpolation inequalities strict. Every $x^\lambda$ thus has an admissible interpolant. As $x^\lambda\to\widehat x$, the LP and the admissible curves have the same objective supremum.

\vspace{-0.4cm}

\paragraph{Bounded support.}

Under PE, $\underline{s}_i$, $\overline{s}_i$, or $y$ can lie outside the support. Take $F\in\mathfrak F_{\PE}(r)$ with support $[\ell,u]$ and choose an exact optimum $r_i\in\mathcal S_F(w_i)\cap[\underline{s}_i,\overline{s}_i]$ for each observation. Assumption~\ref{assump:interior-regime} places $r$ and $r_i$ in the support interior, as are $s_i$ because $0<q_i<1$.

Choose $L$ inside the support below all these prices, and take $H_n\uparrow u$ above them. Keep $U$ on $[L,H_n]$ and apply \eqref{eq:app-left-extension} and \eqref{eq:app-right-extension} outside this interval, using the values and slopes at its endpoints. The resulting full-support distributions $F_n$ preserve the observed demand and every selected optimum $r_i$ and $r$, so they satisfy the same historical PE and current optimality conditions.

If $y<u$, then $\bar F_n(y)=\bar F(y)$ for all sufficiently large $n$. If $u<\infty$ and $y\ge u$, absolute continuity gives $\bar F(y)=0$, and

\vspace{-0.6cm}

\[
0\le\bar F_n(y)\le\bar F(H_n)\longrightarrow0.
\]

Thus $\bar F_n(y)\to\bar F(y)$. Full-support distributions therefore attain the same demand infimum as the original PE ambiguity set.

\vspace{-0.4cm}

\paragraph{Worst-case demand.}
The preceding steps establish, for each $r\in\mathcal R_M$,

\vspace{-0.3cm}

\begin{equation}
\inf_{F\in\mathfrak F_M(r)}\bar F(r+d_M)
=\bar\Gamma_\alpha\bigl(V_M(r)\bigr),
\qquad \bar\Gamma_\alpha(+\infty)=0.
\label{eq:app-fixed-r-equivalence}
\end{equation}

Taking the infimum over feasible optimal prices, with the CE target fixed by Lemma~\ref{lem:current-target-monotonicity} and the PE implementation shift $d_M=\eta_0$, gives

\vspace{-0.3cm}

\[
\begin{aligned}
q_{\min,\alpha}^M(w_0)
&=\inf_{r\in\mathcal R_M}\inf_{F\in\mathfrak F_M(r)}\bar F(r+d_M)\\
&=\inf_{r\in\mathcal R_M}\bar\Gamma_\alpha\bigl(V_M(r)\bigr)
=\bar\Gamma_\alpha\!\left(\sup_{r\in\mathcal R_M}V_M(r)\right),
\end{aligned}
\]

The last equality follows because $\bar\Gamma_\alpha$ is continuous and decreasing, proving \eqref{eq:decision-space-response-set}.
\qed

\vspace{-0.4cm}

\subsection{Proof of Lemma~\ref{lem:e4-grid-error}}
\label{app:e4-grid-proof}

\vspace{-0.3cm}

If the zero-demand test succeeds, the algorithm is exact. Otherwise, \eqref{eq:e4-feasible-grid} is a finite, nonempty subset of $\mathcal R_M$, and
\[
\widehat q_{\min,\alpha}^M(w_0)=\min_{r\in\mathcal G_M(e)}m_M(r)
\ge\inf_{r\in\mathcal R_M}m_M(r)=q_{\min,\alpha}^M(w_0).
\]
For any $r\in I_j$, let $\widetilde r$ be a grid point with $|r-\widetilde r|\le\ell_j/(2N_j)$. By \eqref{eq:e4-demand-lipschitz},

\vspace{-0.3cm}

\[
\widehat q_{\min,\alpha}^M(w_0)
\le m_M(\widetilde r)
\le m_M(r)+\frac{K_j\ell_j}{2N_j}
\le m_M(r)+e.
\]

\vspace{-0.3cm}

All feasible partition points are evaluated directly. Taking the infimum over \eqref{eq:e4-feasible-partition} proves \eqref{eq:e4-demand-error} with finitely many LP evaluations.
\qed

\vspace{-0.3cm}

\subsection{Proof of Proposition~\ref{prop:error-informativeness}}
\label{app:proof-error-informativeness}

For each historical observation, define

\vspace{-0.5cm}

\[
e_i^{\CE}(F)=|\phi_F(s_i)-w_i|,
\qquad
e_i^{\PE}(F)=\dist(s_i,\mathcal S_F(w_i)).
\]
For either model $M$, a tolerance $e$ makes the historical restriction redundant relative to Q exactly when

\vspace{-0.6cm}

\[
e_i^M(F)\le e\quad\text{for every }F\in\mathcal F_\alpha^Q(\mathcal H)
\quad\Longleftrightarrow\quad
e\ge\sup_{F\in\mathcal F_\alpha^Q(\mathcal H)}e_i^M(F).
\]

Thus the two suprema in \eqref{eq:critical-error-main} and \eqref{eq:decision-distance-threshold} are exactly the tolerances at which the corresponding restrictions become redundant.

Moreover, convexity of $U$ gives $g(s_i)\ge u_i/s_i>0$. By \eqref{eq:app-phi-transformed}, every Q-consistent distribution satisfies

\vspace{-0.4cm}

\[
\alpha s_i-\frac{s_i}{u_i}
\le\phi_F(s_i)<s_i.
\]
Hence $\delta_i^{\mathrm{crit}}<\infty$ for every $\alpha\in[0,1]$. For $\alpha>0$, Proposition~\ref{prop:error-model-comparison} also implies that $\eta_i^{\mathrm{crit}}$ is finite.
\qed

\vspace{-0.4cm}

\subsection{Proof of Proposition~\ref{prop:error-model-comparison}}
\label{app:proof-error-comparison}

Let $\alpha>0$ and $F\in\mathcal F_\alpha^{\CE,\delta}(\mathcal H)$. Choose an exact optimal price $r_i$ for each observation. Under Assumption~\ref{assump:interior-regime}, $\phi_F(r_i)=w_i$, and $\alpha$-regularity implies

\vspace{-0.5cm}

\[
\alpha|s_i-r_i|
\le|\phi_F(s_i)-\phi_F(r_i)|
=|\phi_F(s_i)-w_i|
\le\delta_i.
\]
Thus $\dist(s_i,\mathcal S_F(w_i))\le\delta_i/\alpha$ for every $i$, proving \eqref{eq:foc-to-ds-set}.

For the reverse inclusion, take $F\in\mathcal F_\alpha^{\PE,\eta}(\mathcal H)$ and choose a nearest optimum $r_i$ with $|s_i-r_i|\le\eta_i$. The Lipschitz condition gives

\vspace{-0.5cm}

\[
|\phi_F(s_i)-w_i|
=|\phi_F(s_i)-\phi_F(r_i)|
\le L|s_i-r_i|\le L\eta_i.
\]
Applying this to every observation proves \eqref{eq:ds-to-foc-set}.

The condition $\alpha>0$ is necessary for a uniform CE-to-PE conversion. To see this, consider

\vspace{-0.5cm}

\[
\bar F_\kappa(s)=[1+(1-\kappa)s]^{-1/(1-\kappa)},
\qquad s\ge0,\quad0<\kappa<1.
\]

\vspace{-0.3cm}

These distributions satisfy Assumption~\ref{assump:q-feasible} and are regular, with $\phi_{F_\kappa}(s)=\kappa s-1$. For $w_i=1$, $r_i=2/\kappa$ is optimal. Taking $s_i=(2+\delta)/\kappa$ for any fixed $\delta>0$ gives

\vspace{-0.3cm}

\[
|\phi_{F_\kappa}(s_i)-w_i|=\delta,
\qquad
|s_i-r_i|=\frac{\delta}{\kappa}\longrightarrow\infty
\quad\text{as }\kappa\downarrow0.
\tag*{\qedsymbol}
\]

\vspace{-0.4cm}

\subsection{Proof of Theorem~\ref{thm:capacity}}
\label{app:proof-supplier-policy}

\vspace{-0.3cm}

Fix $w\in\mathcal W$ and let $a=q_{\min,\alpha}^M(w)$. By Theorems~\ref{thm:insertion} and~\ref{thm:decision-space-reduction}, the robust objective is

\vspace{-0.3cm}

\[
wa-cz-b(a-z)^+
=
\begin{cases}
(w-b)a+(b-c)z, & 0\le z\le a,\\
wa-cz, & z\ge a.
\end{cases}
\]

\vspace{-0.3cm}

The objective increases up to $z=a$ and decreases thereafter, so $z^*=a$ and $\Pi_\alpha^M(w)=(w-c)a$. Maximizing over $w\in\mathcal W$ proves \eqref{eq:upstream-policy}.
\qed

\vspace{-0.5cm}

\subsection{Proof of Corollary~\ref{cor:e4-policy-error}}
\label{app:e4-policy-proof}

Write $a_w=q_{\min,\alpha}^M(w)$. Since $\widehat q(w)\ge a_w$, Theorem~\ref{thm:capacity} gives

\vspace{-0.5cm}

\[
V_\alpha^M=\max_{w\in\mathcal W}(w-c)a_w
\le\max_{w\in\mathcal W}(w-c)\widehat q(w)
=\widehat V_\alpha^M.
\]

For the selected policy, $\widehat z=\widehat q(\widehat w) \ge a_{\widehat w}$, so emergency sourcing is not needed at the worst-case demand. Its worst-case profit is

\vspace{-0.5cm}

\[
\begin{aligned}
\pi(\widehat w,\widehat z,a_{\widehat w})
&=\widehat w a_{\widehat w}-c\widehat q(\widehat w)\\
&=\widehat V_\alpha^M-\widehat w\bigl(\widehat q(\widehat w)-a_{\widehat w}\bigr)\\
&\ge\widehat V_\alpha^M-\widehat w e_{\widehat w}
\ge\widehat V_\alpha^M-\sigma.
\end{aligned}
\]
The robust optimal value $V_\alpha^M$ bounds the profit of any feasible policy from above, completing \eqref{eq:e4-policy-error}.
\qed

\vspace{-0.7cm}

\section{Response-Price Search and Approximation Error}
\label{app:e4-search}

\vspace{-0.3cm}

This appendix constructs the finite price grid used in Algorithm~\ref{alg:error-global-search}. Fix the data, error parameters, and wholesale price, and write $\theta=\theta_M$, $d=d_M$, and $\rho=1-\alpha$. The fixed price nodes are

\vspace{-0.3cm}

\[
\mathcal T_M=
\begin{cases}
\{s_i:i=1,\ldots,T\}, & M=\CE,\\
\{s_i,\overline{s}_i:i=1,\ldots,T\}\cup\{\underline{s}_i:\underline{s}_i>w_i\}, & M=\PE.
\end{cases}
\]
We use these nodes to partition the response-price range.

\vspace{-0.5cm}

\subsection{Bounding the response-price range}
\label{app:e4-range}

\vspace{-0.2cm}

The range selection is summarized in Algorithm~\ref{alg:e4-range-test}; its two cases are justified below.

\vspace{-0.2cm}

\begin{algorithm}[H]
\small
\caption{Response-price range and zero-demand test}
\label{alg:e4-range-test}
\begin{algorithmic}[1]
\Require $\calH$, $\alpha$, $M$, $\theta_M$, and historical error tolerances, with $\mathcal R_M\ne\varnothing$.
\Ensure A zero-demand certificate or a finite upper bound $\overline r_M$.
\State Form the fixed price nodes $\mathcal T_M$ defined in Appendix~\ref{app:e4-search}.
\If{$\alpha>0$}
    \State $\displaystyle\overline r_M\gets\min_{i=1,\ldots,T}\max\left\{s_i,\frac{\theta_M+s_i/u_i}{\alpha}\right\}$.
    \State \Return $\overline r_M$.
\EndIf
\State $\overline t\gets\max\mathcal T_M$.
\State Check interpolation-consistent feasibility with \eqref{eq:e4-tail-test} replacing \eqref{eq:error-lp-current}.
\If{feasible}
    \State \Return \textsc{Zero demand}.
\Else
    \State \Return $\overline r_M=\overline t$.
\EndIf
\end{algorithmic}
\end{algorithm}

\vspace{-0.3cm}

\paragraph{Positive $\alpha$.}
For $\alpha>0$, let $u_i=\bar\Gamma_\alpha^{-1}(q_i)$ and define

\vspace{-0.5cm}

\[
\overline r_M
=
\min_{i=1,\ldots,T}
\max\left\{s_i,\frac{\theta+s_i/u_i}{\alpha}\right\}.
\]
Every feasible response satisfies $\theta<r\le\overline r_M$. Indeed, convexity gives $g(s_i)\ge u_i/s_i$, hence

\vspace{-0.5cm}

\[
\phi_F(s_i)=s_i-\frac{1+\rho u_i}{g(s_i)}
\ge\alpha s_i-\frac{s_i}{u_i}.
\]
For an interior optimum $r\ge s_i$, $\alpha$-regularity gives
$
\theta=\phi_F(r)
\ge\phi_F(s_i)+\alpha(r-s_i)
\ge\alpha r-\frac{s_i}{u_i}.
$
Thus either $r<s_i$ or $r\le(\theta+s_i/u_i)/\alpha$. Taking the minimum over $i$ gives the stated bound.

\vspace{-0.4cm}

\paragraph{Regular demand.}
For $\alpha=0$, let $\overline t=\max\mathcal T_M$. We first use a feasibility test to identify a case in which the worst-case demand is zero. If the test fails, all feasible response prices are bounded by $\overline t$.

Add
\begin{equation}
(\overline t-\theta)g(\overline t)\le 1+U(\overline t)
\label{eq:e4-tail-test}
\end{equation}
to the historical constraints on $\mathcal T_M$ without the current-response equality, and run the consistency test. If feasible, take an interpolant and write $u=U(\overline t)$, $g=g(\overline t)>0$. When \eqref{eq:e4-tail-test} is strict, choose $k>g$ such that $(\overline t-\theta)k<1+u$. For any sufficiently large $R>\max\{\theta,\overline t\}$, set

\vspace{-0.5cm}

\[
U(R)=(R-\theta)k-1,\qquad g(R)=k.
\]
The secant slope is
\[
D_R=\frac{U(R)-u}{R-\overline t}
=k+\frac{(\overline t-\theta)k-1-u}{R-\overline t},
\]

For large $R$, $g<D_R<k$, so \eqref{eq:app-derivative-interpolant} gives a convex interpolant with $\phi_F(R)=\theta$. If \eqref{eq:e4-tail-test} holds with equality, extend $U$ affinely from $\overline t$ to $R$ instead.

In both cases, the right extension \eqref{eq:app-right-extension} makes $R$ the largest optimal price and satisfies the tail condition, without changing the history. Its demand obeys

\vspace{-0.4cm}

\[
\bar F(R+d)\le\bar F(R)
=\frac{1}{(R-\theta)k}\longrightarrow0.
\]
Hence $q_{\min,0}^M(w_0)=0$.

Conversely, a feasible response $r>\overline t$ must satisfy $\phi_F(\overline t)\le\phi_F(r)=\theta$, which is precisely \eqref{eq:e4-tail-test}. The full-support extension in Appendix~\ref{app:proof-decision-space-reduction} preserves this condition whenever $r>\overline t$. Thus, if the test fails, all feasible responses satisfy $r\le\overline t$; set $\overline r_M=\overline t$.

\vspace{-0.5cm}

\subsection{Feasible partition and grid construction}
\label{app:e4-partition}

With a finite search bound $\overline r_M$, define the fixed-price demand value

\vspace{-0.5cm}

\[
m_M(r)=\bar\Gamma_\alpha(V_M(r)),
\qquad \bar\Gamma_\alpha(+\infty)=0,
\qquad \theta<r\le\overline r_M.
\]

\vspace{-0.2cm}

We construct a finite grid that approximates the minimum demand within $e$. The following result simplifies the feasible-price search.

\begin{lemma}[Interval of feasible response prices]
\label{lem:e4-feasible-interval}
Under Assumptions~\ref{assump:q-feasible} and~\ref{assump:interior-regime}, for $M\in\{\CE,\PE\}$, the set $\mathcal R_M$ is an interval whenever it is nonempty. The interval may consist of a single price.
\end{lemma}

\paragraph{Proof.}
Take $r_1<r_2$ in $\mathcal R_M$. By the extension argument in Appendix~\ref{app:proof-decision-space-reduction}, we can use full-support admissible distributions $F_1,F_2$ that preserve the historical restrictions and have $r_k$ optimal under $\theta_M$. Write $U_k(t)=\bar\Gamma_\alpha^{-1}(\bar F_k(t))$, $g_k=U_k'$, and $\rho=1-\alpha$. For $r_1<r<r_2$, let

\vspace{-0.5cm}

\[
h_k(r)=(r-\theta_M)g_k(r)-\rho U_k(r)-1
       =g_k(r)\bigl(\phi_{F_k}(r)-\theta_M\bigr),\qquad k=1,2.
\]
Since $\phi_{F_k}$ is nondecreasing and $\phi_{F_k}(r_k)=\theta_M$, we have $h_1(r)\ge0$ and $h_2(r)\le0$. Choose $\lambda\in[0,1]$ such that $(1-\lambda)h_1(r)+\lambda h_2(r)=0$ and define

\vspace{-0.5cm}

\[
U_\lambda=(1-\lambda)U_1+\lambda U_2,\qquad
 g_\lambda=(1-\lambda)g_1+\lambda g_2.
\]
The function $U_\lambda$ is convex with positive derivative. Its demand distribution remains admissible: for $\alpha=0$, the tail condition $U_k(t)/t\to\infty$ is preserved by convex combinations; for $\alpha>0$, convexity and positive slopes ensure the required tail decay. Historical demand equalities and the CE or PE restrictions are preserved because their transformed constraints are linear in $U$ and $g$. Moreover, $h_\lambda(r)=0$, so $\phi_{F_\lambda}(r)=\theta_M$ and $r$ is an optimal response. Thus every price between $r_1$ and $r_2$ belongs to $\mathcal R_M$. \qed

Algorithm~\ref{alg:e4-grid-build} uses this interval property to separate feasible-price identification from the partition needed to control the demand approximation error.

\begin{algorithm}[H]
\small
\caption{Feasible response-price grid construction}
\label{alg:e4-grid-build}
\begin{algorithmic}[1]
\Require $\mathcal T_M$, $\theta=\theta_M$, $d=d_M$, finite $\overline r_M$, and accuracy $e>0$, with $\mathcal R_M\ne\varnothing$.
\Ensure A finite feasible price grid $\mathcal G_M(e)$.
\State Identify the feasible interval $\mathcal R_M\subseteq(\theta,\overline r_M]$ using parametric LP feasibility and interpolation-consistency tests.
\If{$\mathcal R_M=\{r_*\}$}
    \State \Return $\mathcal G_M(e)=\{r_*\}$.
\EndIf
\State Split $\mathcal R_M$ at crossings $t,t-d$ for $t\in\mathcal T_M$ and at changes in the parametric LP value branches.
\State Let $I_j=(\underline r_j,\overline r_j)$ be the resulting open price intervals, and let $\mathcal E_M$ contain all feasible partition points.
\State $\mathcal G_M(e)\gets\mathcal E_M$.
\For{each interval $I_j$}
    \State $\ell_j\gets\overline r_j-\underline r_j$.
    \If{$V_M(r)=+\infty$ throughout $I_j$}
        \State $K_j\gets0$.
    \Else
        \State Choose finite $K_j\ge\sup_{r\in I_j}|m_M'(r)|$.
    \EndIf
    \State $N_j\gets\max\{1,\lceil K_j\ell_j/(2e)\rceil\}$.
    \For{$k=1,\ldots,N_j$}
        \State $\displaystyle\mathcal G_M(e)\gets\mathcal G_M(e)\cup\left\{\underline r_j+\frac{(k-\tfrac12)\ell_j}{N_j}\right\}$.
    \EndFor
\EndFor
\State \Return $\mathcal G_M(e)$.
\end{algorithmic}
\end{algorithm}

\vspace{-0.8cm}

\paragraph{Partitioning the feasible prices.}

Lemma~\ref{lem:e4-feasible-interval} shows that $\mathcal R_M$ has no separated feasible components: it is either a single point or an interval. The single-point case is evaluated directly. For an interval, its boundaries and their feasibility can be determined by the parametric LP feasibility and interpolation-consistency tests. The remaining partition is needed to track changes in the value $V_M(r)$, not to search for isolated feasible prices.

The order of the LP nodes can change only when $r$ or $r+d$ crosses a fixed node. The possible ordering changes are
\[
\mathcal D_M
=\{\theta,\overline r_M\}
\cup\Bigl(\{t,t-d:t\in\mathcal T_M\}\cap(\theta,\overline r_M)\Bigr).
\]
Within each ordering interval, the nodal LP coefficients depend affinely on $r$. The interpolation patterns in \eqref{eq:app-interpolation-criterion} provide finitely many systems: for any chosen affine intervals, impose $L_j=R_j=0$ there and apply the strict-feasibility test \eqref{eq:app-strict-witness} elsewhere. By the mixing argument in Appendix~\ref{app:proof-decision-space-reduction}, the associated closed LP has the same supremum as its admissible interpolants.

To identify the feasible interval boundaries and the changes in $V_M$, write these LPs in standard form $A(r)x=b(r)$, $x\ge0$. The roots of their matrix minors mark possible rank changes. Between these roots, basis solutions $A_B(r)^{-1}b(r)$ and objective values are rational in $r$. LP unboundedness is checked through
\[
\max\{c(r)^\top v:A(r)v=0,\ v\ge0,\ \mathbf1^\top v=1\}.
\]
Refining at changes in basis feasibility, strict-feasibility tests, unboundedness, and objective ordering gives a finite partition of $\mathcal R_M$. On each open interval $I_j$, either $V_M(r)=+\infty$ throughout or $V_M(r)=P_j(r)/Q_j(r)$ for polynomials $P_j,Q_j$ with $Q_j(r)\ne0$. Let $\mathcal E_M$ contain the feasible interval endpoints and all feasible partition points, evaluated with coincident nodes merged. Then

\vspace{-0.4cm}

\begin{equation}
\mathcal R_M=\mathcal E_M\cup\bigcup_{j=1}^J I_j.
\label{eq:e4-feasible-partition}
\end{equation}
When $\mathcal R_M$ is a singleton, $J=0$ and $\mathcal E_M=\mathcal R_M$.

\vspace{-0.4cm}

\paragraph{Controlling demand variation.}
On an interval with finite $V(r)=V_M(r)$, differentiating $m_M(r)=\bar\Gamma_\alpha(V(r))$ gives
\[
m'(r)=
\begin{cases}
-V'(r)[1+\rho V(r)]^{-1/\rho-1},&\rho>0,\\
-V'(r)e^{-V(r)},&\rho=0.
\end{cases}
\]

On $I_j$, the nonnegative rational function $V$ has either a finite limit or a pole at each endpoint $a$. At a pole, $V(r)\sim C|r-a|^{-k}$ for $C>0$ and an integer $k\ge1$, giving

\vspace{-0.5cm}

\[
|m'(r)|=O\!\left(|r-a|^{k/\rho-1}\right)
\quad(0<\rho\le1),
\]

This is bounded because $k/\rho\ge1$. When $\rho=0$, $m'(r)\to0$ instead. Hence $m_M'$ is bounded on each $I_j$. A valid $K_j\ge\sup_{r\in I_j}|m_M'(r)|$ follows by checking the endpoint limits and the interior stationary points of $m_M'$:
\[
(1+\rho V)V''-(1+\rho)(V')^2=0.
\]

After clearing denominators, this is a polynomial equation. If it holds identically, $m_M'$ is constant. If $V_M=+\infty$ throughout $I_j$, take $K_j=0$. The mean value theorem now gives

\vspace{-0.5cm}

\begin{equation}
|m_M(r)-m_M(r')|\le K_j|r-r'|,
\qquad r,r'\in I_j.
\label{eq:e4-demand-lipschitz}
\end{equation}

\vspace{-0.5cm}

\paragraph{Constructing the grid.}
For each feasible interval $I_j=(\underline r_j,\overline r_j)$ of length $\ell_j$, choose

\vspace{-0.3cm}

\[
N_j=\max\left\{1,\left\lceil\frac{K_j\ell_j}{2e}\right\rceil\right\}.
\]
Take $N_j$ equally spaced midpoints in each $I_j$, together with all feasible partition points:

\vspace{-0.4cm}

\begin{equation}
\mathcal G_M(e)
=\mathcal E_M\cup
\bigcup_{j=1}^J
\left\{\underline r_j+\frac{(k-\tfrac12)\ell_j}{N_j}:
 k=1,\ldots,N_j\right\}.
\label{eq:e4-feasible-grid}
\end{equation}

All grid points are feasible. Every $r\in I_j$ is within $\ell_j/(2N_j)$ of a midpoint, so \eqref{eq:e4-demand-lipschitz} limits the demand error to $e$. The grid is finite because the partition and all $N_j$ are finite.

\begin{remark}[Role of the grid construction]
The partition and grid construction above serves to establish the finite-accuracy guarantee in Lemma~\ref{lem:e4-grid-error}. In practice, the outer price search may use simpler one-dimensional methods with fixed-$r$ LP evaluations.
\end{remark}

\vspace{-0.6cm}

\section{Computation of the CE and PE Informativeness Thresholds}
\label{app:threshold-computation}

\vspace{-0.3cm}

We compute the critical tolerances in Proposition~\ref{prop:error-informativeness} over the Q ambiguity set, using $u_h=\bar\Gamma_\alpha^{-1}(q_h)$ and the transformation from Appendix~\ref{app:proof-history-feasible}.

\vspace{-0.3cm}

\subsection{Computing the CE threshold}

By \eqref{eq:app-phi-transformed},
$
\delta_i^{\mathrm{crit}}
=\sup_{F\in\mathcal F_\alpha^Q(\mathcal H)}
\left|s_i-w_i-\frac{1+\rho u_i}{g(s_i)}\right|.
$
Since $s_i-(1+\rho u_i)/g$ increases with $g>0$, only the smallest and largest feasible slopes at $s_i$ matter. Write $U_h=U(s_h)$, $g_h=g(s_h)$, and $\Delta_h=s_{h+1}-s_h$. For $\sigma\in\{-1,1\}$, solve

\vspace{-1cm}

\begin{subequations}
\label{eq:ce-threshold-lp}
\begin{align}
J_i(\sigma)
=
\sup_{U,g\ge0}\quad
&\sigma g_i
\label{eq:ce-threshold-objective}\\
\text{s.t.}\quad
&U_h=u_h,
&&h=1,\ldots,T,
\label{eq:ce-threshold-data}\\
&\Delta_hg_h
\le
U_{h+1}-U_h
\le
\Delta_hg_{h+1},
&&h=1,\ldots,T-1,
\label{eq:ce-threshold-shape}\\
&U_1\le s_1g_1,
\label{eq:ce-threshold-support}\\
&\Delta_hg_h
=
U_{h+1}-U_h
=
\Delta_hg_{h+1},
&&h\in\mathcal J_\alpha^Q.
\label{eq:ce-threshold-flat}
\end{align}
\end{subequations}

\vspace{-0.3cm}

The Q interpolation-consistency test from Appendix~\ref{app:proof-decision-space-reduction} determines $\mathcal J_\alpha^Q$. Its interpolation argument gives
$
g_i^-=-J_i(-1)=\inf_{F\in\mathcal F_\alpha^Q(\mathcal H)}g(s_i),
g_i^+=J_i(1)=\sup_{F\in\mathcal F_\alpha^Q(\mathcal H)}g(s_i).
$
Consequently,

\vspace{-0.1cm}

\begin{equation}
\delta_i^{\mathrm{crit}}
=
\max
\left\{
w_i-s_i+\frac{1+\rho u_i}{g_i^-},
\;
s_i-w_i-\frac{1+\rho u_i}{g_i^+}
\right\}.
\label{eq:ce-threshold-value}
\end{equation}

Here $g_i^-\ge u_i/s_i>0$ by convexity, and $1/g_i^+=0$ when $g_i^+=+\infty$.

\vspace{-0.3cm}

\subsection{Computing the PE threshold}

\vspace{-0.3cm}

Under Assumption~\ref{assump:interior-regime}, exact optimal prices satisfy $\phi_F(r)=w_i$. By monotonicity, a support-interior price $t$ lies between $s_i$ and its nearest optimum precisely when $(s_i-t)(\phi_F(t)-w_i)>0$. Hence
$
\dist(s_i,\mathcal S_F(w_i))=\sup\bigl(\{0\}\cup\{\,|t-s_i|:t>w_i,0<\bar F(t)<1,(s_i-t)(\phi_F(t)-w_i)>0\,\}\bigr).
$
Prices between $s_i$ and an interior optimum satisfy $0<\bar F(t)<1$. Equation~\eqref{eq:app-phi-transformed} rewrites the sign condition as

\vspace{-0.3cm}

\begin{equation}
(s_i-t)
\left[
(t-w_i)g(t)-\rho U(t)-1
\right]
>0.
\label{eq:pe-distance-sign}
\end{equation}

For each $t>w_i$, add $t$ to the historical grid and merge duplicate nodes. Write $\tau_1<\cdots<\tau_m$ for the resulting prices, $U_j=U(\tau_j)$, $g_j=g(\tau_j)$, and $\Delta_j=\tau_{j+1}-\tau_j$. Solve

\vspace{-0.8cm}

\begin{subequations}
\label{eq:pe-threshold-lp}
\begin{align}
Z_i(t)
=
\sup_{U,g\ge0}\quad
&
(s_i-t)
\left[
(t-w_i)g(t)-\rho U(t)-1
\right]
\label{eq:pe-threshold-objective}\\
\text{s.t.}\quad
&U(s_h)=u_h,
&&h=1,\ldots,T,
\label{eq:pe-threshold-data}\\
&\Delta_jg_j
\le
U_{j+1}-U_j
\le
\Delta_jg_{j+1},
&&j=1,\ldots,m-1,
\label{eq:pe-threshold-shape}\\
&U_1\le\tau_1g_1,
\label{eq:pe-threshold-support}\\
&\Delta_jg_j
=
U_{j+1}-U_j
=
\Delta_jg_{j+1},
&&j\in\mathcal J_\alpha^Q(t).
\label{eq:pe-threshold-flat}
\end{align}
\end{subequations}

Use the interpolation-consistency test from Appendix~\ref{app:proof-decision-space-reduction} to determine $\mathcal J_\alpha^Q(t)$. Then $Z_i(t)>0$ exactly when some Q-consistent distribution satisfies \eqref{eq:pe-distance-sign}. Indeed, a positive LP objective is preserved by sufficiently small mixing with a strictly feasible vector, which yields an admissible interpolant. Conversely, any admissible distribution satisfying the strict sign condition gives a positive LP value; bounded-support curves can be approximated as in Appendix~\ref{app:proof-decision-space-reduction}.

Taking the supremum over these test prices gives
$
\eta_i^{\mathrm{crit}}
=
\sup
\left(
\{0\}
\cup
\left\{
|t-s_i|:
t>w_i,\,
Z_i(t)>0
\right\}
\right).
$
When finite, the threshold can be approximated by a one-dimensional search using \eqref{eq:pe-threshold-lp}.

\end{document}